\documentclass{article}
\usepackage{amssymb}
\usepackage{amsmath}
\usepackage{amsthm}
\usepackage{amsfonts}

=155mm \textheight =215mm

\begin{document}
\def\mathscr{\mathcal}

\newtheorem{theorem}{Theorem}[section]
\newtheorem{proposition}[theorem]{Proposition}
\newtheorem{lemma}[theorem]{Lemma}
\newtheorem{corollary}[theorem]{Corollary}
\newtheorem{conjecture}[theorem]{Conjecture}
\theoremstyle{remark}
\newtheorem{remark}[theorem]{Remark}

\title {$\ell$-adic GKZ hypergeometric sheaves and exponential sums
\thanks{Part of this paper was worked out during my visit of IHES.
I would like to thank IHES for its hospitality. I am especially
grateful to O. Gabber for pointing out many subtle points which I
ignored. My research is supported by the NSFC (10525107).}}

\author {Lei Fu\\
{\small Chern Institute of Mathematics and LPMC, Nankai University,
Tianjin 300071, P. R. China}\\
{\small leifu@nankai.edu.cn, leisfu@gmail.com}}

\date{}
\maketitle

\begin{abstract}
To a torus action on a complex vector space, Gelfand, Kapranov and
Zelevinsky introduce a system of differential equations, called the
GKZ hypergeometric system. Its solutions are GKZ hypergeometric
functions. We study the $\ell$-adic counterpart of the GKZ
hypergeometric system, which we call the $\ell$-adic GKZ
hypergeometric sheaf. It is an object in the derived category of
$\ell$-adic sheaves on the affine space over a finite field. Traces
of Frobenius on stalks of this object at rational points of the
affine space define the hypergeometric functions over the finite
field introduced by Gelfand and Graev. We prove that the $\ell$-adic
GKZ hypergeometric sheaf is perverse, calculate its rank, and prove
that it is irreducible under the non-resonance condition. We also
study the weight filtration of the GKZ hypergeometric sheaf,
determine its lisse locus, and apply our result to the study of
weights of twisted exponential sums.



\end{abstract}

\section*{Introduction}

Let $k$ be a finite field with $q$ elements of characteristic $p$,
and let $\ell$ be a prime number distinct from $p$. Throughout this
paper, we work with a nontrivial additive character $\psi:
k\to\overline {\mathbb Q}_\ell^\ast$. Let $\chi_1,\ldots,
\chi_m:k^\ast\to\overline {\mathbb Q}_\ell^\ast$ be multiplicative
characters, and let $f, f_1,\ldots, f_m\in k[t_1^{\pm 1}\ldots,
t_n^{\pm1}]$ be Laurent polynomials. In number theory, we often need
to study the mixed character sum
$$S_1=\sum_{t_1,\ldots, t_n\in k^\ast} \chi_1(f_1(t_1,\ldots,
t_n))\cdots \chi_m(f_m(t_1,\ldots, t_n))\psi(f(t_1,\ldots, t_n)).$$
Consider the twisted exponential sum
\begin{eqnarray*}
S_2 &=&\sum_{t_1,\ldots, t_{n+m}\in k^\ast}
\chi_1^{-1}(t_{n+1})\cdots \chi_m^{-1}(t_{n+m})\\
&&\qquad\qquad\qquad \psi\left(f(t_1,\ldots,
t_n)+t_{n+1}f_1(t_1,\ldots, t_n)+\cdots +t_{n+m}f_m(t_1,\ldots,
t_n)\right).
\end{eqnarray*}
One can show if $\chi_i$ $(i=1,\ldots, m)$ are nontrivial, then
$$S_2=g(\chi_1^{-1},\psi)\cdots g(\chi_m^{-1},\psi)S_1,$$ where
$$g(\chi_i^{-1},\psi)=\sum_{x\in k^\ast} \chi_i^{-1}(x)\psi(x)$$
are Gauss sums. As Gauss sums are well-understood, the study of
$S_1$ is reduced to the study of $S_2$.

Let $$A=\left(\begin{array}{ccc} w_{11}&\cdots&w_{1N}\\
\vdots&&\vdots\\
w_{n1}&\cdots&w_{nN}\end{array} \right)$$ be an $(n\times N)$-matrix
of rank $n$ with integer entries. Denote the column vectors of $A$
by ${\mathbf w}_1,\ldots, {\mathbf w}_N\in \mathbb Z^n$. It defines
an action of the $n$-dimensional torus $\mathbb T_{\mathbb
Z}^n=\mathrm{Spec}\, \mathbb Z[t_1^{\pm 1},\ldots, t_n^{\pm 1}]$ on
the $N$-dimensional affine space $\mathbb A_{\mathbb
Z}^N=\mathrm{Spec}\, \mathbb Z[x_1,\ldots, x_N]$:
$$
\mathbb T_{\mathbb Z}^n\times \mathbb A_{\mathbb Z}^N \to \mathbb
A_{\mathbb Z}^N, \quad \Big((t_1,\ldots, t_n),(x_1,\ldots,
x_N)\Big)\mapsto (t_1^{w_{11}}\cdots t_n^{w_{n1}}x_1,\ldots,
t_1^{w_{1N}}\cdots t_n^{w_{nN}}x_N).
$$
Let $\chi_1,\ldots, \chi_n:k^\ast\to\overline {\mathbb Q}_\ell^\ast$
be multiplicative characters, and let $a_j\in k$. Write the twisted
exponential sum $S_2$ in the form
$$\sum_{t_1,\ldots, t_n\in k^\ast} \chi_1(t_1)\cdots
\chi_n(t_n)\psi\left(\sum_{j=1}^N a_j t_1^{w_{1j}}\cdots
t_n^{w_{nj}}\right).$$ In the case where $\chi_1,\ldots, \chi_n$ are
trivial, Denef-Loeser (\cite{DL}) and Adolphson-Sperber (\cite{AS1},
\cite{AS2}) study this exponential sum. The general case is treated
in \cite{AS3} and \cite{Fu}.

A method emphasized by Gelfand-Kapranov-Zelevinsky and called the
``$A$-philosophy" in the book \cite[5.1]{GKZ3} is that instead of
studying the above twisted exponential sum for a fixed $(a_1,\ldots,
a_N)$, one could treat $(a_1,\ldots, a_N)$ as indeterminate, or
equivalently, one treats $(a_1,\ldots, a_N)$ as parameters and study
the corresponding family of twisted exponential sums. In \cite{GG}
and \cite{GGR}, Gelfand and Graev define the \emph{hypergeometric
function over the finite field} to be
$$\mathrm{Hyp}_\psi(x_1,\ldots, x_N; \chi_1,\ldots, \chi_n)
=\sum_{t_1,\ldots, t_n\in k^\ast}\chi_1(t_1)\cdots
\chi_n(t_n)\psi\Big( \sum_{j=1}^N x_j t_1^{w_{1j}}\cdots
t_n^{w_{nj}}\Big).$$ Its value at $(a_1,\ldots, a_N)$ is the twisted
exponential sum introduced above.

\bigskip
Let $(\gamma_1,\ldots, \gamma_n)\in \mathbb C^n$ be a fixed
parameter. In \cite{GKZ1}, Gelfand, Kapranov and Zelevinsky define
the \emph{$A$-hypergeometric system} to be the system of
differential equations
\begin{eqnarray*}
&&\sum_{j=1}^N w_{ij} x_j\frac{\partial f}{\partial x_j}+\gamma_i
f=0 \quad (i=1,\ldots, n),\\
&& \prod_{a_j>0} \left(\frac{\partial}{\partial x_j}\right)^{a_j}f=
\prod_{a_j<0} \left(\frac{\partial}{\partial x_j}\right)^{-a_j}f,
\end{eqnarray*}
where for the second system of equations, $(a_1,\ldots,
a_N)\in\mathbb Z^N$ goes over the family of integral linear
relations
$$\sum_{j=1}^N
a_j{\mathbf w}_j=0$$ among ${\mathbf w}_1,\ldots, {\mathbf w}_N$.
Holomorphic solutions of the $A$-hypergeometric systems are called
\emph{$A$-hypergeometric functions}. We often call the
$A$-hypergeometric system as the \emph{GKZ hypergeometric system}.
Integral representations of solutions of the GKZ hypergeometric
system are of the form
$$f(x_1,\ldots, x_N, \gamma_1,\ldots, \gamma_n)=\int_C
t_1^{\gamma_1}\cdots t_n^{\gamma_n}e^{\sum_{j=1}^N x_j
t_1^{w_{1j}}\cdots
t_n^{w_{nj}}}\frac{dt_1}{t_1}\cdots\frac{dt_n}{t_n},$$ where $C$ is
a cycle in the complex torus $\mathbb T_{\mathbb C}^n$. (Confer
\cite[Corollary 2 in \S 4.2]{GGR1}). The hypergeometric function
over finite field is an arithmetic analogue of the above integral.
For this reason, we also call $\mathrm{Hyp}_\psi(x_1,\ldots, x_N;
\chi_1,\ldots, \chi_n)$ as the \emph{GKZ hypergeometric sum}.

Note that $\mathrm{Hyp}_\psi(x_1,\ldots, x_N; \chi_1,\ldots,
\chi_n)$ is homogeneous with respect to the torus action in the
sense that
\begin{eqnarray*}
&&\mathrm{Hyp}_\psi(t_1^{w_{11}}\cdots t_n^{w_{n1}}x_1,\ldots,
t_1^{w_{1N}}\cdots t_n^{w_{nN}}x_N; \chi_1,\ldots, \chi_n)\\
&=&
\chi_1^{-1}(t_1)\cdots\chi_n^{-1}(t_n)\mathrm{Hyp}_\psi(x_1,\ldots,
x_N; \chi_1,\ldots, \chi_n)
\end{eqnarray*}
for any $t_1,\ldots, t_n\in k^\ast$ and $x_1,\ldots, x_N\in k$.
Hypergeometric sums over finite fields defined by Katz in
\cite[8.2.7]{K1} are special cases of the GKZ hypergeometric sum.
Indeed, let $A$ be the $((n+m-1)\times(n+m))$-matrix
$$A=\left(I_{n+m-1}, \mathbf{w}_{n+m} \right),$$ where $I_{n+m-1}$
is the identity matrix of size $n+m-1$, and $\mathbf {w}_{n+m}$ is
the transpose of the vector
$$(\underbrace{-1,\ldots,-1}_{n-1}, \underbrace{1, \ldots, 1}_{m}).$$
Then the GKZ hypergeometric sum associated to this matrix evaluated
at
$$(x_1,\ldots, x_{n+m})=(\underbrace{1,\ldots,1}_{n-1},
\underbrace{-1, \ldots, -1}_{m}, x)$$ is
\begin{eqnarray*}
&&\mathrm{Hyp}_\psi(\underbrace{1,\ldots,1}_{n-1}, \underbrace{-1,
\ldots, -1}_{m}, x;\chi_1,\ldots,
\chi_{n+m-1})\\
&=& \sum_{t_1,\ldots, t_{n+m-1}}\chi_1(t_1)\cdots
\chi_{n+m-1}(t_{n+m-1})\psi\left(t_1+\cdots +t_{n-1}-t_n-\cdots
-t_{n+m-1}+\frac{x t_n\cdots t_{n+m-1}}{t_1\cdots t_{n-1}}\right)
\end{eqnarray*}
The last expression is exactly Katz's hypergeometric sum
$\mathrm{Hyp}(\psi,\chi_1,\ldots, \chi_{n-1},1;\chi_n^{-1},\ldots,
\chi_{n+m-1}^{-1})(x)$. So Katz's hypergeometric sum can be
expressed in terms of the GKZ hypergeometric sum. Conversely, using
the homogeneity property of the GKZ hypergeometric sum with respect
to the torus action, one can also express the GKZ hypergeometric sum
associated to the above matrix $A$ evaluated at an arbitrary point
$(x_1,\ldots, x_{n+m})$ in terms of the (one-variable) Katz's
hypergeometric sum. An interesting special case of Katz's
hypergeometric sum is the Kloosterman sum
$$\mathrm {Kl}_\psi(\chi_1,\ldots, \chi_n,1)(x)=\sum_{t_1,\ldots, t_n}\chi_1(t_1)
\cdots \chi_n(t_n)\psi\left(t_1+\cdots+t_n+\frac{x}{t_1\ldots
t_n}\right).$$ It is the GKZ hypergeometric sum
$\mathrm{Hyp}_\psi(x_1,\ldots, x_{n+1};\chi_1,\ldots,\chi_n)$
associated to the $(n\times (n+1))$-matrix
$$A=\left(\begin{array}{cccc}
1&&& -1\\
&\ddots&&\vdots\\&&1&-1
\end{array}\right)$$
evaluated at $(x_1,\ldots, x_{n+1})=(1,\ldots, 1,x)$.

We say $A$ satisfies the \emph{nonconfluence condition} if there
exist integers $c_1,\ldots, c_n\in\mathbb Z$ such that
$$\sum_{i=1}^n c_iw_{ij}=1\quad (j=1,\ldots, N),$$
that is, $\mathbf w_j$ $(j=1,\ldots, N)$ lie in the hyperplance
$\sum_{i=1}^n c_iw_i=1.$ If $A$ is the $((n+m-1)\times(n+m))$-matrix
corresponding to Katz's hypergeometric sum
$\mathrm{Hyp}(\psi,\chi_1,\ldots, \chi_{n-1},1;\chi_n^{-1},\ldots,
\chi_{n+m-1}^{-1})(x)$ described above, then $A$ satisfies the
nonconfluence condition if and only if $m=n$.

\begin{theorem}[Gelfand-Kapranov-Zelevinsky, \cite{GKZ1, GKZ2}]
Suppose that $A$ satisfies the nonconfluence condition, that
${\mathbf w}_1,\ldots, {\mathbf w}_N$ generate $\mathbb Z^n$, and
that the ring $\mathbb C[t_1^{w_{11}}\cdots t_n^{w_{n1}}, \ldots,
t_1^{w_{1N}}\cdots t_n^{w_{nN}}]$ is normal. Let $\Delta$ be the
convex hull of $\{0,{\mathbf w}_1,\ldots, {\mathbf w}_N\}$.

(i) The GKZ hypergeometric system is holonomic.

(ii) The dimension of the space of GKZ hypergeometric functions at a
generic point is $n!\mathrm{vol}(\Delta)$.

(iii) If $(\gamma_1,\ldots, \gamma_n)$ satisfies the so-called
non-resonance condition, then the sheaf of GKZ hypergeometric
functions defines an irreducible local system on a Zariski dense
open subset of $\mathbb C^N$.
\end{theorem}

We refer the reader to \cite[\S2.9]{GKZ2} for the definition of the
non-resonance condition.

A theorem of Hotta (\cite[\S II 6]{H}) says that if $A$ satisfies
nonconfluence condition, then the GKZ hypergeometric system is
regular holonomic. In \cite{A}, Adolphson studies the GKZ-system
without the nonconfluence condition, and he proves the following.

\begin{theorem}[Adolphson, \cite{A}] ${}$

(i) The GKZ hypergeometric system is holonomic.

(ii) Suppose $(\gamma_1,\ldots, \gamma_n)$ satisfies the so-called
semi-nonresonance condition. Then the dimension of the space of GKZ
hypergeometric functions at a generic point is
$n!\mathrm{vol}(\Delta)/[\mathbb Z^n: M']$, where $\Delta$ is the
convex hull of $\{0,{\mathbf w}_1,\ldots, {\mathbf w}_n\}$, and $M'$
is the subgroup of $\mathbb Z^n$ generated by $\{{\mathbf
w}_1,\ldots, {\mathbf w}_N\}$.
\end{theorem}

We refer the reader to \cite[pg. 284]{A} for the definition of the
semi-nonresonance condition.

Adolphson conjectures that even without the nonconfluence condition,
the sheaf of GKZ hypergeometric functions defines an irreducible
local system on a Zariski dense open subset of $\mathbb C^N$ if
$(\gamma_1,\ldots, \gamma_n)$ satisfies the non-resonance condition.

In this paper, we introduce the $\ell$-adic GKZ hypergeometric
sheaf, and we prove theorems of Gelfand-Kapranov-Zelevinsky and
Adolphson, and verify Adolphson's conjecture in this context.
Moreover, we study the weight filtration of the GKZ hypergeometric
sheaf. Specializing our result to a rational point, we recover the
main results in \cite{AS1, AS2, AS3, DL, Fu} about the weights of
exponential sums.

Denote by $\mathbb A^N$ (resp. $\mathbb T^n$) the $N$-dimensional
affine space (resp. $n$-dimensional torus) over the finite field
$k$. Let $\mathscr L_\psi$ be the Artin-Schreier sheaf on $\mathbb
A^1$ associated to the nontrivial additive character $\psi$. For any
character $\chi:(k^\ast)^n\to\overline{\mathbb Q}_\ell^\ast$, let
$\mathscr K_\chi$ be the Kummer sheaf on $\mathbb T^n$ associated to
$\chi$. (Confer \cite[Sommes trig. 1.7]{SGA 4 1/2} for the
definition and properties of the Kummer sheaf and the Artin-Schreier
sheaf). Let
$$\pi_1:\mathbb T^n\times \mathbb A^N\to \mathbb T^n, \quad
\pi_2:\mathbb T^n\times \mathbb A^N\to \mathbb A^N$$ be the
projections, and let $F$ be the morphism
$$F:\mathbb T^n\times \mathbb A^N\to
\mathbb A^1,\quad (t_1,\ldots, t_n, x_1,\ldots, x_N)\mapsto
\sum_{j=1}^Nx_jt_1^{w_{1j}}\cdots t_n^{w_{nj}}.$$ We define the
\emph{$\ell$-adic GKZ hypergeometric sheaf} associated to the Kummer
sheaf $\mathscr K_\chi$ and the vectors $\mathbf w_1,\ldots, \mathbf
w_N$ in $\mathbb Z^n$ to be the object in the derived category
$D_c^b(\mathbb A^N, \overline{\mathbb Q}_\ell)$ of $\ell$-adic
sheaves on $\mathbb A^N$ given by
$$\mathrm{Hyp}_{\psi}(\chi)=R\pi_{2!}\Big(\pi_1^\ast\mathscr K_\chi\otimes F^\ast \mathscr
L_\psi\Big)[n+N].$$ (Confer \cite[1.1.2]{D} for the definition of
the derived category $D_c^b(\mathbb A^N,\overline{\mathbb Q}_\ell)$
of $\ell$-adic sheaves.) If we take $\mathscr K_\chi=\mathscr
K_{\chi_1}\boxtimes\cdots\boxtimes \mathscr K_{\chi_n}$, then by the
Grothendieck trace formula (\cite[Rapport 3.2]{SGA 4 1/2}), for any
$x=(x_1,\ldots, x_N)\in \mathbb A^N(k)$, we have
$$\mathrm{Hyp}_\psi(x_1,\ldots, x_N;\chi_1,\ldots,\chi_n)=
(-1)^{n+N}\mathrm{Tr}\Big(\mathrm
{Frob}_x,\big(\mathrm{Hyp}_{\psi}(\chi)\big)_{\bar x}\Big),$$ where
$\mathrm{Frob}_x$ is the geometric Frobenius at $x$. We have the
following.

\begin{theorem} Let $\chi:(k^\ast)^n\to\overline{\mathbb
Q}_\ell$ be a character.

(i) $\mathrm{Hyp}_\psi(\chi)$ is a mixed perverse sheaf on $\mathbb
A^N$ of weights $\leq n+N$.

(ii) Suppose $\mathbf{w}_1,\ldots, \mathbf{w}_N$ generate $\mathbb
Z^n$, and suppose $\chi$ satisfies the non-resonance condition
defined below. Then $\mathrm{Hyp}_\psi(\chi)$ is an irreducible pure
perverse sheaf of weight $n+N$.
\end{theorem}

The non-resonance condition is defined as follows. Let $\delta$ be
the convex polyhedral cone in $\mathbb R^n$ generated by $\mathbf
w_1,\ldots, \mathbf w_N$. For any proper face $\tau$ of $\delta$,
let $\mathbb T_{\tau}$ be the torus $\mathrm{Spec}\, k[\mathbb
Z^n\cap \mathrm{span}\,\tau].$ Note that $\mathbb Z^n/\mathbb
Z^n\cap \mathrm{span}\,\tau$ is torsion free and hence free. So we
have
$$\mathbb Z^n\cong (\mathbb Z^n\cap \mathrm{span}\,\tau)\oplus (\mathbb Z^n/\mathbb Z^n\cap \mathrm{span}\,\tau).$$ The
inclusion of $\mathbb Z^n\cap \mathrm{span}\,\tau$ in $\mathbb Z^n$
induces a homomorphism of tori
$$p_\tau: \mathbb T^n=\mathrm{Spec}\,
k[\mathbb Z^n]\to \mathbb T_{\tau}=\mathrm{Spec}\, k[\mathbb Z^n\cap
\mathrm{span}\,\tau].$$ Its kernel is isomorphic to the torus
$\mathrm{Spec}\, k[\mathbb Z^n/\mathbb Z^n\cap
\mathrm{span}\,\tau].$ We say $\chi$ satisfies the
\emph{non-resonance condition} if for any proper face $\tau$ of
$\delta$, the restriction of $\mathscr K_\chi$ to
$\mathrm{ker}\,p_\tau$ is nontrivial. This is equivalent to saying
that $\mathscr K_\chi$ is not of the form $p_\tau^\ast \mathscr K$
for any Kummer sheaf $\mathscr K$ on $\mathbb T_\tau$. Since any
proper face of $\delta$ is contained in a codimension one face, to
check the non-resonance condition, it suffices to work with those
proper faces $\tau$ of $\delta$ of codimension one.

Theorem 0.3 (ii) was also obtained by T. Terasoma (\cite{T}) based
on a different notion of the nonresonance condition.

\medskip Let $f=\sum_{j=1}^N a_j t_1^{w_{1j}}\cdots t_n^{w_{nj}}$
be a Laurent polynomial such that $a_j\in\bar k$ are all nonzero,
and let $\Delta$ be the convex hull in $\mathbb R^n$ of the set
$\{0,\mathbf w_1,\ldots, \mathbf w_N\}$. We say $f$ is
\emph{nondegenerate with respect to $\Delta$} if for any face
$\Gamma$ of $\Delta$ not containing the origin, the subscheme of
$\mathbb T_{\bar k}^n$ defined by
$$\frac{\partial f_\Gamma}{\partial t_1}=\cdots =\frac{\partial
f_\Gamma}{\partial t_n}=0$$ is empty, where $f_\Gamma=\sum_{\mathbf
w_j\in\Gamma} a_j t_1^{w_{1j}}\cdots t_n^{w_{nj}}$.

Let $\Lambda$ be a subset of $\mathbb Z^n$, and let
$$P(t_1,\ldots, t_n)=\sum_{\mathbf w=(w_1,\ldots, w_n)\in\Lambda}
x_{\mathbf w} t_1^{w_1}\cdots t_n^{w_n}$$ considered as a Laurent
polynomial with variable coefficients $x_{\mathbf w}$ ($\mathbf
w\in\Lambda$). Let's recall the definition of the
$\Lambda$-discriminant $\Delta_\Lambda(x_{\mathbf w})$. (Confer
\cite{GKZ3} \S 9.1 Definition 1.2.) Let $\nabla_0\subset \mathbb
C^\Lambda$ be the set of those points $(x_{\mathbf w})_{\mathbf w\in
\Lambda}\in \mathbb C^\Lambda$ for which there exists
$t^{(0)}=(t_1^{(0)}, \ldots, t_n^{(0)})$ in the torus $(\mathbb
C^\ast)^n$ such that
$$P(t^{(0)})=\frac{\partial P}{\partial t_1}(t^{(0)})=\cdots=
\frac{\partial P}{\partial t_n}(t^{(0)})=0.$$ Let $\nabla_\Lambda$
be the Zariski closure of $\nabla_0$ in $\mathbb C^\Lambda$. It is
an irreducible variety defined over $\mathbb Q$, and it is conical,
that is, it is invariant under multiplication by scalars. If
$\nabla_\Lambda$ is a subvariety of $\mathbb C^\Lambda$ of
codimension $1$, we define the \emph{$\Lambda$-discriminant
$\Delta_\Lambda(x_{\mathbf w})$} to be an irreducible polynomial
with integer coefficients in the variables $x_{\mathbf w}$ ($\mathbf
w\in \Lambda$) which vanishes exactly on $\nabla_\Lambda$. Such a
polynomial is uniquely determined up to sign. If the codimension of
$\nabla_\Lambda$ is larger than $1$, we set
$\Delta_\Lambda(x_{\mathbf w})=1$.

Suppose there exists an affine hyperplane in $\mathbb Q^n$ not
containing the origin such that all points in $\Lambda$ lie in this
hyperplane. Then there exist integers $c_1,\ldots, c_n\in \mathbb Z$
and a nonzero integer $c$ such that
$$\sum_{i=1}^n c_iw_i=c$$ for all $\mathbf w=(w_1,\ldots, w_n)$ in
$\Lambda$. We have
$$P(t^{c_1} t_1, \ldots, t^{c_n} t_n)=t^c P(t_1,\ldots, t_n).$$
Applying $\frac{d}{dt}|_{t=1}$ to both sides of this equation, we
get
$$\sum_{i=1}^n c_i\frac{\partial P}{\partial t_i}(t_1,\ldots,
t_n)=cP(t_1,\ldots, t_n).$$ So the condition $$\frac{\partial
P}{\partial t_1}(t^{(0)})=\cdots= \frac{\partial P}{\partial
t_n}(t^{(0)})=0$$ implies the condition $$P(t^{(0)})=0.$$ In
particular, the hypersurface $\Delta_\Lambda=0$ is the closure of
the set consisting of points  $(x_{\mathbf w})_{\mathbf w\in
\Lambda}\in \mathbb C^\Lambda$ for which there exists
$t^{(0)}=(t_1^{(0)}, \ldots, t_n^{(0)})$ in the torus $(\mathbb
C^\ast)^n$ such that
$$\frac{\partial P}{\partial t_1}(t^{(0)})=\cdots=
\frac{\partial P}{\partial t_n}(t^{(0)})=0,$$ provided that this
closure has codimension $1$.

Consider the Laurent polynomial
$$F=\sum_{j=1}^N x_jt_1^{w_{1j}}\cdots t_n^{w_{nj}}$$ with variable
coefficients $x_j$. For any face $\Gamma$ of $\Delta$ not containing
the origin, let
$$F_\Gamma=\sum_{\mathbf w_j\in \tau} x_jt_1^{w_{1j}}\cdots t_n^{w_{nj}}.$$
Note that $\Gamma$ lies in an affine hyperplane in $\mathbb Q^n$ not
containing the origin. Let $V$ be the complement of the hypersurface
$$\Big(\prod_{0\not\in\Gamma\prec \Delta} \Delta_{\Gamma\cap\{\mathbf w_1,\ldots,
\mathbf w_N\}}\Big)(x_1,\ldots, x_N)=0$$ in $\mathbb T^N_{\bar k}$.
By the above discussion, for any $(a_1,\ldots, a_N)\in V(\bar k)$,
the Laurent polynomial $f=\sum_{j=1}^N a_j t_1^{w_{1j}}\cdots
t_n^{w_{nj}}$ is nondegenerate with respect to $\Delta$. If the
coefficients of the polynomial $\prod_{0\not\in\Gamma\prec \Delta}
\Delta_{\Gamma\cap\{\mathbf w_1,\ldots, \mathbf w_N\}}$ are not all
divisible by $p$, a condition which holds for sufficiently large
$p$, then $V$ is a Zariski dense open subset of $\mathbb A^N$.

\medskip
For any convex polyhedral cone $\delta$ in ${\mathbb R}^n$ with $0$
being a face, define the convex polytope ${\rm poly}(\delta)$ to be
the intersection of $\delta$ with a hyperplane in ${\mathbb R}^n$
which does not contain $0$ and intersects each one dimensional face
of $\delta$. Note that ${\rm poly}(\delta)$ is defined only up to
combinatorial equivalence. For any convex polytope $\Delta$ in
${\mathbb R}^n$ and any face $\Gamma$ of $\Delta$, define ${\rm
cone}_\Delta(\Gamma)$ to be the cone generated by $u'-u$
($u'\in\Delta$, $u\in \Gamma$), and define ${\rm
cone}_\Delta^\circ(\Gamma)$  to be the image of ${\rm
cone}_\Delta(\Gamma)$ in ${\mathbb R}^n/{\rm span}(\Gamma-\Gamma)$.
Note that $0$ is a face of ${\rm cone}_\Delta^\circ(\Gamma)$. We
define polynomials $\alpha(\delta)$ and $\beta(\Delta)$ in one
variable $T$ inductively by the following formulas:
\begin{eqnarray*}
\alpha(\{0\})&=&1,\\
\beta(\Delta)&=&(T^2-1)^{{\rm dim}(\Delta)}+\sum_{\Gamma
\prec\Delta,\;
\Gamma\not=\Delta}(T^2-1)^{{\rm dim}(\Gamma)}\alpha({\rm cone}_{\Delta}^\circ(\Gamma)),\\
\alpha(\delta)&=&{\rm trunc}_{\leq {\rm
dim}(\delta)-1}((1-T^2)\beta({\rm poly}(\delta))),
\end{eqnarray*}
where ${\rm trunc}_{\leq d}(\cdot)$ denotes taking the degree $\leq
d$ part of a polynomial. These polynomials are first introduced by
Stanley \cite{S}. Note that $\alpha(\delta)$ and $\beta(\Delta)$
only involve even powers of $T$, and they depend only on the
combinatorial types of $\delta$ and $\Delta$.

\medskip Let $\chi:\mathbb T_k^n(k)\to \overline {\mathbb Q}_l^\ast$ be a
character, let $\Delta$ be the convex hull of $\{0,\mathbf
w_1,\ldots,\mathbf w_N\}$, let $\delta$ be the convex polyhedral
cone generated by $\mathbf w_1,\ldots,\mathbf w_N$, and let $T$ be
the set of faces $\tau$ of $\delta$ so that $\tau\not=\delta$ and
$\mathscr K_\chi\cong p_\tau^\ast\mathscr K_{\chi_\tau}$ for a
Kummer sheaf $\mathscr K_{\chi_\tau}$ on the torus $\mathbb
T_\tau={\rm Spec}\, k[{\mathbb Z}^n\cap {\rm span}\,\tau]$, where
$$p_\tau:\mathbb T_k^n=\mathrm{Spec}\, k[{\mathbb Z}^n]\to \mathbb T_\tau=
\mathrm{Spec}\, k[\mathbb Z^n\cap \mathrm{span}\,\tau]$$ is the
morphism induced by the inclusion $k[\mathbb Z^n\cap
\mathrm{span}\,\tau] \hookrightarrow k[\mathbb Z^n].$ Define
$$e(\Delta,\chi)= (-1)^Nn!{\rm vol}(\Delta)+\sum_{\tau\in T}
(-1)^{n-{\rm dim}(\tau)+N}({\rm dim}(\tau))!{\rm vol}(\Delta\cap
\tau)\alpha({\rm cone}_{\delta}^\circ(\tau))(1)$$ and define a
polynomial $E(\Delta,\chi)$ inductively by
$$E(\Delta,\chi)=e(\Delta,\chi)T^{n+N}-\sum_{\tau\in T} (-1)^{n-{\rm
dim}(\tau)+N-N_\tau}T^{N-N_\tau}E(\Delta\cap\tau,\chi_\tau)\alpha({\rm
cone}_{\delta}^\circ(\tau)),$$ where $N_\tau$ is the number of those
$\mathbf w_j$ $(j=1,\ldots, N)$ lying in $\tau$,
$\mathrm{cone}_\delta(\tau)$ is the cone generated by $u'-u$
($u'\in\delta$, $u\in \tau$), and $\mathrm{cone}_\delta^\circ(\tau)$
is the image of $\mathrm{cone}_\delta(\tau)$ in ${\mathbb R}^n/{\rm
span}(\tau)$. Note that we have
$\mathrm{cone}_\delta(\tau)=\mathrm{cone}_{\Delta}(\Delta\cap\tau)$
and
$\mathrm{cone}_\delta^\circ(\tau)=\mathrm{cone}_{\Delta}^\circ(\Delta\cap
\tau)$.

\begin{theorem} Let
$\Delta$ be the convex hull of $\{0,\mathbf{w}_1,\ldots,
\mathbf{w}_N\}$. Suppose there exists a Zariski dense open subset
$V$ of $\mathbb T^N$ such that for any $(a_1,\ldots, a_N)\in V(\bar
k)$, the Laurent polynomial $f=\sum_{j=1}^N a_j t_1^{w_{1j}}\cdots
t_n^{w_{nj}}$ is nondegenerate with respect to $\Delta$.

(i) The rank of $\mathrm{Hyp}_\psi(\chi)$ is
$(-1)^Nn!\mathrm{vol}(\Delta)$. Here for any object $K$ in
$D_c^b(\mathbb A^N, \overline{\mathbb Q}_\ell)$, the rank of $K$ is
defined to be $\sum_{i}(-1)^i \mathrm{dim}\,\mathscr H^i(K)_{\bar
\eta}$, where $\eta$ is the generic point of $\mathbb A^N$.

(ii) Let
$$P(\mathrm{Hyp}_\psi(\chi))=\sum_{w\in\mathbb Z} e_w T^w,$$ where
$e_w$ is the rank of the weight $w$ sub-quotient of the weight
filtration for the mixed perverse sheaf $\mathrm{Hyp}_\psi(\chi)$.
(Confer \cite[5.3.5]{BBD} for the definition of the weight
filtration of a mixed perverse sheaf). Then we have
\begin{eqnarray*}
P(\mathrm{Hyp}_\psi(\chi))&=&E(\Delta,\chi),\\
e_{n+N}&=&e(\Delta,\chi).
\end{eqnarray*}

(iii) For each codimension 1 face $\Gamma$ of $\Delta$ not
containing the origin, choose relatively prime integers $d_1,\ldots,
d_n$ such that the restriction to $\Delta$ of linear function
$$\phi(v_1,\ldots, v_n)=d_1v_1+\ldots+d_nv_n$$ takes its minimum
$d_\Gamma$ exactly on $\Gamma$. Suppose $p$ does not divide
$d_\Gamma$ for each codimension one face $\Gamma$ of $\Delta$ not
containing the origin. Then for each $i$, $\mathscr
H^i(\mathrm{Hyp}_\psi(\chi))|_V$ is lisse.
\end{theorem}

The main tools to prove our results are the Deligne-Fourier
transformation and toric geometry. Let $\mathbb A'^N$ be the dual
affine space of $\mathbb A^N$. Recall that the Deligne-Fourier
transformation
$$FT_{\psi}: D_c^b(\mathbb A'^N,\overline{\mathbb Q}_\ell)\to D_c^b(\mathbb
A^N,\overline {\mathbb Q}_\ell)$$ associated to the additive
character $\psi$ is the functor defined by
$$FT_{\psi}(K)=Rp_{2!}(p_1^\ast K\otimes \langle\,,\,\rangle^\ast \mathscr
L_{\psi})[N]$$ for any $K\in \mathrm{ob}\,D_c^b(\mathbb
A'^N,\overline{\mathbb Q}_\ell)$, where
$$p_1: \mathbb A'^N\times\mathbb A^N\to\mathbb A'^N,\quad
p_2:\mathbb A'^N\times\mathbb A^N\to\mathbb A^N$$ are the
projections, and $\langle\,,\,\rangle$ is the morphism
$$\langle\,,\,\rangle:\mathbb A'^N\times\mathbb A^N\to\mathbb
A^1,\quad ((\xi_1,\ldots, \xi_N), (x_1,\ldots, x_N))\mapsto
\sum_{j=1}^N x_j\xi_j.$$ Confer \cite{L} for properties of the
Deligne-Fourier transformation. Consider the morphism
$$\iota:\mathbb T^n\to\mathbb A'^N, \quad (t_1,\ldots, t_n)\mapsto
(t_1^{w_{11}}\cdots t_n^{w_{n1}}, \ldots, t_1^{w_{1N}}\cdots
t_n^{w_{nN}}).$$

\begin{theorem}
The morphism $\iota$ is quasi-finite and affine, and we have

(i) $\mathrm{Hyp}_\psi(\chi)\cong FT_{\psi}(\iota_!\mathscr
K_{\chi}[n]).$

(ii) The weight $n+N$ subquotient of the weight filtration for the
mixed perverse sheaf $\mathrm{Hyp}_\psi(\chi)$ is isomorphic to
$FT_{\psi}(\iota_{!\ast}\mathscr K_{\chi}[n])$, and its rank is
equal to $e(\Delta,\chi)$.
\end{theorem}

Let's deduce some corollaries from the above theorems.

\begin{corollary} If $0$ lies in the interior of $\Delta$, or equivalently, if
$\delta$ coincides with $\mathbb R^n$, then
$\mathrm{Hyp}_\psi(\chi)$ is pure of weight $n+N$.
\end{corollary}

\begin{proof}
We can factorize $\iota$ as the composite
$$\mathbb T^n=\mathrm{Spec}\,k[\mathbb Z^n]\stackrel{j}\hookrightarrow  \mathrm{Spec}\,k[\mathbb
Z^n\cap \delta]\stackrel {g}\to \mathbb A^N,$$ where $j$ is the open
immersion of the open dense torus in the affine toric scheme
$\mathrm{Spec}\,k[\mathbb Z^n\cap \delta]$ and it is induced by the
inclusion $k[\mathbb Z^n\cap \delta]\hookrightarrow k[\mathbb Z^n]$,
and $g$ is the morphism induced by the $k$-algebra homomorphism
$$k[\xi_1,\ldots,\xi_N]\to k[\mathbb Z^n\cap\delta], \quad \xi_j\mapsto
\mathbf w_j.$$ One can verify that $g$ is a finite morphism. In the
case where $\delta=\mathbb R^n$, $j$ is an isomorphism. It follows
that $\iota$ is finite, and $\iota_!=\iota_{!\ast}$. So
$\mathrm{Hyp}_\psi(\chi)\cong FT_{\psi}(\iota_{!\ast}\mathscr
K_{\chi}[n])$ is pure of weight $n+N$.
\end{proof}

\begin{corollary} Suppose the condition of Theorem 0.4 (iii) holds.

(i) For $i\not=-N$, we have $\mathscr
H^i(\mathrm{Hyp}_\psi(\chi))|_{V}=0$ and $\mathscr
H^i(FT_{\psi}(\iota_{!\ast}\mathscr K_{\chi}[n]))|_{V}=0$, and
$\mathscr H^{-N}(\mathrm{Hyp}_\psi(\chi))|_{V}$ and $\mathscr
H^{-N}(FT_{\psi}(\iota_{!\ast}\mathscr K_{\chi}[n]))|_{V}$ are lisse
sheaves on $V$ of ranks $n!\mathrm{vol}(\Delta)$ and
$(-1)^Ne(\Delta,\chi)$, respectively.

(ii) If $0$ lies in the interior of $\Delta$, then $\mathscr
H^{-N}(\mathrm{Hyp}_\psi(\chi))|_{V}$ is lisse pure of weight $n$
and of rank $n!\mathrm{vol}(\Delta)$.
\end{corollary}

\begin{proof} (i) follows from Theorems 0.3 (i), 0.4 (iii) and 0.5 (ii).
(ii) follows from Corollary 0.6.
\end{proof}

Specializing to a point $a=(a_1,\ldots, a_n)\in V(k)$, the above
results imply \cite[Theorems 1.3, 1.8]{DL} and \cite[Theorem
0.2]{Fu} under the extra assumption that $p$ is sufficiently large.

\medskip Corresponding to Hotta's result about the regularity of
the GKZ system under the nonconfluence condition, we have the
following theorem which shows that $\mathrm{Hyp}_\psi(\chi)$ is tame
under the nonconfluence condition if $p$ is sufficiently large.

\begin{theorem} Suppose the matrix $A$ satisfies the nonconfluence
condition, that is, there exist integers $c_1,\ldots, c_n\in\mathbb
Z$ such that $\sum_{i=1}^nc_iw_{ij}=1$ $(j=1,\ldots, N)$. Let
$\chi'=\chi_1^{c_1}\cdots \chi_n^{c_n}$.

(i) Suppose $\chi'$ is nontrivial. Let $G(\chi',\psi)$ be the rank
one lisse $\overline{\mathbb Q}_\ell$-sheaf on $\mathrm{Spec}\,k$ so
that the geometric Frobenius acts by multiplication of the Gauss sum
$-g(\chi',\psi)=-\sum_{t\in k^\ast}\chi'(t)\psi(t)$, and we denote
the inverse image of $G(\chi',\psi)$ on any $k$-scheme also by
$G(\chi',\psi)$. Then there exists an object $\mathscr H$ in the
derived category $D_c^b(\mathbb A^N_{\mathbb
Z[1/\ell]},\overline{\mathbb Q}_\ell)$ of $\overline {\mathbb
Q}_\ell$-sheaves on the affine space $\mathbb A^N_{\mathbb
Z[1/\ell]}$ over $\mathbb Z[1/\ell]$ such that
$\mathrm{Hyp}_\psi(\chi)$ is a direct factor of $(\mathscr
H|_{\mathbb A^N_{k}})\otimes G(\chi',\psi)$. In particular, over the
algebraic closure $\bar k$ of $k$,
$\mathrm{Hyp}_\psi(\chi)|_{\mathbb A^N_{\bar k}}$ is a direct factor
of $\mathscr H|_{\mathbb A^N_{\bar k}}$.

(ii) If $\chi'$ is trivial, then there exists an object $\mathscr H$
in $D_c^b(\mathbb A^N_{\mathbb Z[1/\ell]},\overline{\mathbb
Q}_\ell)$ such that $\mathrm{Hyp}_\psi(\chi)$ is a direct factor of
$\mathscr H|_{\mathbb A^N_{k}}$.

In both cases, $\mathscr H$ is independent of the choice of the
additive character $\psi$.
\end{theorem}

Finally we consider the case where ${\mathbf w}_1,\ldots, {\mathbf
w}_N$ all lie in the first quadrant of $\mathbb R^n$, that is
$w_{ij}\geq 0$ for all $i\in\{1,\ldots, n\}, j\in\{1,\ldots, N\}$.
In this case, it is natural to consider the family of exponential
sums
$$\mathrm{Hyp}^{0}_\psi(x_1,\ldots, x_N)=\sum_{t_1,\ldots, t_n\in
k}\psi\Big(\sum_{j=1}^N x_jt_1^{w_{1j}}\cdots t_n^{w_{nj}}\Big).$$
Here unlike the hypergeometric function over the finite field, the
summation is taken over all rational points of the whole affine
space $\mathbb A^n$, not just the torus $\mathbb T^n$. Let
$\pi_2:\mathbb A^n\times\mathbb A^N\to \mathbb A^N$ be the
projection, and let $G:\mathbb A^n\times\mathbb A^N\to\mathbb A^1$
be the morphism $G(t_1,\ldots, t_n,x_1,\ldots, x_N)=
\sum_{j=1}^Nx_jw_1^{w_{1j}}\cdots t_n^{w_{nj}}.$ Define an object in
$D_c^N(\mathbb A^N,\overline{\mathbb Q}_\ell)$ by
$$\mathrm{Hyp}_\psi^0=R\pi_{2!} G^\ast L_\psi[n+N].$$ By the
Grothendieck trace formula, for any $x=(x_1,\ldots, x_N)\in \mathbb
A^N(k)$, we have
$$\mathrm{Hyp}_\psi^0(x_1,\ldots, x_N)=
(-1)^{n+N}\mathrm{Tr}\Big(\mathrm
{Frob}_x,\big(\mathrm{Hyp}_{\psi}^0\big)_{\bar x}\Big),$$ where
$\mathrm{Frob}_x$ is the geometric Frobenius at $x$.

\begin{corollary}
Suppose $\delta$ coincides with the positive quadrant
$\{(w_i)\in\mathbb R^n|w_i\geq 0\}$ of $\mathbb R^n$. Then
$\mathrm{Hyp}_\psi^0$ is isomorphic to the weight $n+N$ subquotient
of the weight filtration for the mixed perverse sheaf
$\mathrm{Hyp}_\psi(\chi)$ with $\chi=1$.
\end{corollary}

\begin{proof} Keep the notation in the proof of Corollary 0.6.
In the case where $\delta$ is the positive quadrant of $\mathbb R^n$
and $\chi=1$, $j$ coincides with the canonical open immersion
$\mathbb T^n\hookrightarrow \mathbb A^n$, $g$ is the morphism
$$g:\mathbb A^n\to\mathbb A'^N,\quad (t_1,\ldots, t_n)\mapsto (t_1^{w_{11}}\cdots t_n^{w_{1n}},\ldots,
t_1^{w_{1N}}\cdots t_n^{w_{1N}}),$$ and
$$\iota_{!\ast}\mathcal K_{\chi}[n]\cong g_! j_{!\ast}\overline
{\mathbb Q}_\ell[n]\cong g_!  \overline {\mathbb Q}_\ell[n].$$
Similarly to Theorem 0.5 (i), one can prove
$\mathrm{Hyp}_{\psi}^0\cong FT(g_!\overline{\mathbb Q}_\ell)[n]$.
(Confer the proof of Lemma 1.1.) Our assertion then follows from
Theorem 0.5 (ii).
\end{proof}

\begin{corollary} Suppose the condition of Theorem 0.4 (iii) holds,
and suppose $\delta$ coincides with the positive quadrant of
$\mathbb R^n$. Then $\mathscr H^i(\mathrm{Hyp}_\psi^0)|_{V}=0$ for
$i\not=N$, and $\mathscr H^{-N}(\mathrm{Hyp}_\psi^0)|_{V}$ is a
lisse sheaf pure of weight $n$ and of rank
$\sum_{0\in\tau\prec\Delta}(-1)^{n-\mathrm{dim}\,\tau}(\mathrm{dim}\,\tau)!\mathrm{vol}(\tau)$.
\end{corollary}

In \cite[8.1]{GL} and \cite[3.1]{Lo}, Gabber and Loeser introduce
hypergeometric sheaves on tori. Let $\pi_i:\mathbb T^1\to \mathbb
T^N$ $(i=1,\ldots, n)$ be homomorphisms of tori defined by
$$\pi_i(t)=(t^{w_{i1}},\ldots, t^{w_{iN}})$$ for some integers $w_{ij}$,
let $\chi_i$ $(i=1,\ldots, n)$ be multiplicative characters of $k$,
and let $\lambda=(\lambda_1,\ldots, \lambda_N)$ be a $k$-point of
$\mathbb T^N$. For any $K,L\in D_c^b(\mathbb T^N, \overline{\mathbb
Q}_\ell)$, the convolution product $K\ast_! L$ of $K$ and $L$ is
defined to be
$$K\ast_! L=Rm_!(K\boxtimes L),$$ where $m:\mathbb T^N\times \mathbb T^N\to\mathbb T^N$
is the multiplication on $\mathbb T^N$. In loc. cit. the
!-hypergeometric sheaf associated to the above data is defined to be
$$\mathrm{Hyp}(!,\lambda, (\pi_i),(\chi_i))=\delta_\lambda \ast_! R\pi_{1!} (\mathcal K_{\chi_1}\otimes
\mathcal L_\psi) \ast_!\cdots\ast_! R\pi_{n!}(\mathcal
K_{\chi_n}\otimes \mathcal L_\psi).$$ For any $k$-points
$x=(x_1,\ldots, x_N)$ of $\mathbb T^N$, by Grothendieck's trace
formula, we have
\begin{eqnarray*}
&&\mathrm{Tr}\Big(\mathrm{Frob}_x, (\mathrm{Hyp}(!,\lambda, (\pi_i),(\chi_i)))_{\bar x}\Big)\\
&=& \sum_{t_1,\ldots, t_n\in k^\ast, \lambda_1t_1^{w_{11}}\cdots
t_n^{w_{n1}}=x_1,\ldots, \lambda_Nt_1^{w_{1N}}\cdots
t_n^{w_{nN}}=x_N} \chi_1(t_1)\cdots \chi_n(t_n) \psi(t_1+\cdots
+t_n).
\end{eqnarray*} I don't know the exact relationship
between the hypergeometric sheaves of Gabber and Loeser and the GKZ
hypergeometric sheaves. In some very special cases, I can express
these two types of hypergeometric sheaves in terms of each other.

\medskip
The paper is organized as follows. In \S1, we first prove Theorem
0.5 (i) and deduce Theorem 0.3 from it. Based on the work of
Denef-Loeser \cite{DL} and systemically using the Deligne-Fourier
transformation, we prove Theorems 0.4 (i)-(ii) and 0.5 (ii). In \S
2, we study the liss locus of $\mathrm{Hyp}_\psi(\chi)$ and prove
Theorem 0.4 (iii). Our method can be considered as a generalization
of Deligne's method of estimating exponential sums
(\cite[3.7.2-3]{D}). In \S3, we prove Theorem 0.8.

\section{Deligne-Fourier Transformation and the GKZ Hypergeometric Sheaf}

Recall that $\mathbb A'^N$ is the dual affine space of $\mathbb A^N$
and $\iota$ is the morphism
$$\iota:\mathbb T^n\to\mathbb A'^N, \quad (t_1,\ldots, t_n)\mapsto
(t_1^{w_{11}}\cdots t_n^{w_{n1}}, \ldots, t_1^{w_{1N}}\cdots
t_n^{w_{nN}}).$$ By abuse of notation, we denote the morphism
$\mathbb T^n\to\mathbb T^N$ induced by $\iota$ also by $\iota$. The
following is Theormem 0.5 (i).

\begin{lemma} The morphism $\iota$ is quasi-finite and affine, and we have
$$FT_{\psi}(\iota_!\mathscr K_{\chi}[n])\cong
\mathrm{Hyp}_\psi(\chi).$$
\end{lemma}

\begin{proof} Since $\mathbb T^n$ and $\mathbb A'^N$ are affine schemes,
$\iota$ is an affine morphism. Recall that the matrix $A=(w_{ij})$
has rank $n$. So over $\mathbb Q$, the vectors $\mathbf w_1,\ldots,
\mathbf w_N$ generate $\mathbb Q^n$. This implies that there exists
an $(N\times n)$-matrix $B=(v_{ij})$ with integer entries such that
$AB=dI_n$ for some nonzero integer $d$, where $I_n$ is the identity
matrix. Consider the morphism
$$\iota': \mathbb T^N\to \mathbb T^n,\quad (\xi_1,\ldots, \xi_N)\mapsto
(\xi_1^{v_{11}}\cdots \xi_N^{v_{N1}},\ldots, \xi_1^{v_{1n}}\cdots
\xi_N^{v_{Nn}}).$$ One can verify that the composite $\iota'\iota$
coincides with the morphism
$$\mathbb T^n\to\mathbb T^n, \quad (t_1,\ldots, t_n)\to
(t_1^d,\ldots, t_n^d).$$ So $\iota'\iota$ is a finite morphism. This
implies that $\iota:\mathbb T^n\to \mathbb T^N$ is finite, and hence
$\iota:\mathbb T^n\to\mathbb A'^N$ is quasi-finite.

Fix notation by the following commutative diagram, where all squares
are Cartesian:
$$\begin{array}{ccccc}
\mathbb T^n\times \mathbb A^N&\stackrel{\iota\times \mathrm
{id}}\to&\mathbb A'^N\times\mathbb A^N&
\stackrel{p_2}\to&\mathbb A^N\\
{\scriptstyle {\pi_1}}\downarrow&&{\scriptstyle
p_1}\downarrow&&\downarrow\\
\mathbb T^n&\stackrel{\iota}\to&\mathbb A'^N&\to&\mathrm{Spec}\,k.
\end{array}$$
By the proper base change theorem and the projection formula, we
have
\begin{eqnarray*}
FT_\psi(\iota_! \mathscr K_\chi[n]) &\cong& Rp_{2!} \big(p_1^\ast
\iota_!\mathscr K_\chi\otimes
\langle\,,\rangle^\ast \mathscr L_\psi\big)[n+N]\\
&\cong&  Rp_{2!}\big((\iota\times\mathrm{id})_! \pi_1^\ast\mathscr
K_\chi\otimes
\langle\,,\rangle^\ast \mathscr L_\psi\big)[n+N]\\
&\cong& Rp_{2!}(\iota\times\mathrm{id})_! \big(\pi_1^\ast\mathscr
K_\chi\otimes (\iota\times\mathrm{id})^\ast\langle\,,\rangle^\ast
\mathscr L_\psi\big)[n+N].
\end{eqnarray*}
We have $p_2(\iota\times\mathrm{id})=\pi_2$ and
$\langle\,,\rangle\circ (\iota\times\mathrm{id})=F$, where
$\pi_2:\mathbb T^n\times \mathbb A^N\to \mathbb A^N$ is the
projection and $F:\mathbb T^n\times \mathbb A^N\to \mathbb A^1$ is
the morphism defined by
$$F(t_1,\ldots, t_n, x_1,\ldots,
x_N)=\sum_{j=1}^Nx_jt_1^{w_{1j}}\cdots t_n^{w_{nj}}.$$ So we have
$$FT_\psi(\iota_! \mathscr K_\chi[n])\cong R\pi_{2!}(\pi_1^\ast \mathscr K_\chi\otimes F^\ast\mathscr
L_\psi)[n+N] =\mathrm{Hyp}_\psi(\chi).$$
\end{proof}

To proceed, we need some facts about toric varieties. The reader can
consult \cite{CLS} or \cite{F} for proof. Let $\Sigma$ be a fan in
the dual vector space $(\mathbb R^n)^\ast$ of $\mathbb R^n$. Denote
the toric variety over $k$ associated to the fan $\Sigma$ by
$X(\Sigma)$. It is covered by affine open subschemes
$U_\sigma=\mathrm{Spec}\, k[\mathbb Z^n\cap \check \sigma]$, where
$\sigma$ goes over cones in $\Sigma$ and $\check\sigma$ is the dual
cone in $\mathbb R^n$ of $\sigma$. For the cone $\sigma=0$, we get
the open dense torus $\mathbb T^n=\mathrm{Spec}\, k[\mathbb Z^n]$ of
$X(\Sigma)$. The torus action of $\mathbb T^n$ on itself can be
extended to an action of $\mathbb T^n$ on $X(\Sigma)$. Let
$O_\sigma$ be the torus $\mathrm{Spec}\, k[\mathbb Z^n\cap
\sigma^\perp]$. We have a $k$-epimorphism
$$k[\mathbb Z^n\cap \check\sigma]\to k[\mathbb Z^n\cap
\sigma^\perp], \quad u\mapsto \left\{\begin{array}{cl} u&\hbox{if } u\in \sigma^\perp\\
0&\hbox{if }u\not\in \sigma^\perp\end{array}\right.\hbox{for any }
u\in \mathbb Z^n\cap \check \sigma.$$ It induces a closed immersion
$$O_\sigma \to U_\sigma.$$ So we can regard $O_\sigma$ as a
subscheme of $X(\Sigma)$. We have
$X(\Sigma)=\coprod_{\sigma\in\Sigma} O_\sigma$, and this gives the
orbital decomposition of the toric variety $X(\Sigma)$ under the
torus action.

Suppose $\delta$ is a rational convex polyhedral cone of dimension
$n$ in $\mathbb R^n$. Then the set $\Sigma(\delta)$ of faces of the
dual cone $\check \delta$ is a fan in $(\mathbb R^n)^\ast$. For any
face $\tau$ of $\delta$, let $\mathrm{cone}_\delta(\tau)$ be the
convex polyhedral cone in $\mathbb R^n$ generated by $u'-u$ ($u'\in
\delta$, $u\in \tau)$, and let $(\mathrm{cone}_\delta(\tau))^\vee$
be its dual cone in $(\mathbb R^n)^\ast$. We have
$$(\mathrm{cone}_\delta(\tau))^\vee=\check \delta\cap \tau^\perp,$$
and the map $\tau\mapsto (\mathrm{cone}_\delta(\tau))^\vee$ defines
a one-to-one correspondence between faces of $\delta$ and faces of
$\check \delta$. So we have
$$\Sigma(\delta)=\{(\mathrm{cone}_\delta(\tau))^\vee|\tau\prec\delta\}.$$
The toric variety $X(\Sigma(\delta))$ is affine, and is just
$\mathrm{Spec}\, k[\mathbb Z^n\cap \delta]$, and it has an open
covering $\mathrm{Spec}\, k[\mathbb Z^n\cap
\mathrm{cone}_\delta(\tau)]$ $(\tau\prec\delta)$. Suppose
furthermore that $0$ is a face of $\delta$. Then we have a
$k$-epimorphism
$$k[\mathbb Z^n \cap \delta]\to k,
\quad
u\mapsto \left\{\begin{array}{cl} 1&\hbox{if } u=0\\
0&\hbox{if }u\not=0\end{array}\right.\hbox{for any }u\in \mathbb Z^n
\cap \delta.$$ It induces a closed immersion
$$x_0:\mathrm{Spec}\, k\to \mathrm{Spec}\,
k[\mathbb Z^n\cap \delta]=X(\Sigma(\delta)).$$ We call $x_0$ the
\emph{distinguished point} in $X(\Sigma(\delta))$. It is fixed under
the torus action on $X(\Sigma(\delta))$.

\begin{lemma} Let $\Sigma$ be a fan in $\mathbb R^n$, let $\sigma\in \Sigma$, and
let $\delta_\sigma$ be the image of the dual cone $\check \sigma$
under the canonical homomorphism $\mathbb R^n \to \mathbb
R^n/\sigma^\perp$. Note that $\delta_\sigma$ is a rational convex
polyhedral cone in $\mathbb R^n/\sigma^\perp$ of dimension ${\rm
dim}(\mathbb R^n/\sigma^\perp)$, and $0$ is a face of
$\delta_\sigma$. Let $$x_\sigma:\mathrm{Spec}\,k\to
X(\Sigma(\delta_\sigma))=\mathrm{Spec}\, k[(\mathbb Z^n/\mathbb
Z^n\cap \sigma^\perp)\cap \delta_\sigma]$$ be the distinguished
point in $X(\Sigma(\delta_\sigma))$, let
$$j:{\mathbb T}^n\hookrightarrow X(\Sigma),\;j_\sigma:{\mathbb T}^{{\rm
dim}(\sigma)}=\mathrm{Spec}\, k[\mathbb Z^n/\mathbb Z^n\cap
\sigma^\perp]\hookrightarrow X(\Sigma(\delta_\sigma))$$ be the
immersions of the open dense tori in toric schemes. Denote by
$$p_\sigma:{\mathbb T}^n={\rm Spec}\, k[\mathbb Z^n]\to O_\sigma=
{\rm Spec}\, k [\mathbb Z^n\cap \sigma^\perp]$$ the morphism induced
by the inclusion $k[\mathbb Z^n\cap \sigma^\perp]\hookrightarrow
k[\mathbb Z^n].$

(i)  If $\mathscr K_\chi$ is not the inverse image under $p_\sigma$
of any Kummer sheaf on $O_\sigma$, then $(j_{!\ast}(\mathscr
K_\chi[n]))|_{O_\sigma}$ and $(Rj_{\ast}(\mathscr
K_\chi[n]))|_{O_\sigma}$ are acyclic, where $O_\sigma$ is considered
as a subscheme of $X(\Sigma)$.

(ii) If $\mathscr K_\chi=p_\sigma^\ast \mathscr K_{\chi_\sigma}$ for
some Kummer sheaf $\mathscr K_{\chi_\sigma}$ on $O_\sigma$, then
\begin{eqnarray*}
(j_{!\ast}(\mathscr K_\chi[n]))|_{O_\sigma}&\cong & (\mathscr
K_{\chi_\sigma}[n-{\rm dim}(\sigma)])\otimes x_\sigma^\ast
({j_\sigma}_{!\ast}(\overline {\mathbb Q}_\ell[{\rm
dim}(\sigma)])),\\
(Rj_{\ast}(\mathscr K_\chi[n]))|_{O_\sigma}&\cong & (\mathscr
K_{\chi_\sigma}[n-{\rm dim}(\sigma)])\otimes x_\sigma^\ast
({Rj}_{\sigma\ast}(\overline {\mathbb Q}_\ell[{\rm dim}(\sigma)])),
\end{eqnarray*}
where by abuse of notation we denote the inverse image on any
$k$-scheme of the complex $x_\sigma^\ast
({j_\sigma}_{!\ast}(\overline {\mathbb Q}_\ell[{\rm dim}(\sigma)]))$
(resp. $x_\sigma^\ast ({Rj}_{\sigma\ast}(\overline {\mathbb
Q}_\ell[{\rm dim}(\sigma)])))$ on $\mathrm{Spec}\, k$ also by
$x_\sigma^\ast ({j_\sigma}_{!\ast}(\overline {\mathbb Q}_\ell[{\rm
dim}(\sigma)]))$ (resp. $x_\sigma^\ast ({Rj}_{\sigma\ast}(\overline
{\mathbb Q}_\ell[{\rm dim}(\sigma)])))$.
\end{lemma}

\begin{proof} The statements about $(j_{!\ast}(\mathscr
K_\chi[n]))|_{O_\sigma}$ is \cite[Lemma 1.3]{Fu}. The same proof
works for $(Rj_{\ast}(\mathscr K_\chi[n]))|_{O_\sigma}$.
\end{proof}

In the following, we apply Lemma 1.2 to the special case where
$\Sigma=\Sigma(\delta)$ for a rational convex polyhedral cone
$\delta$ of dimension $n$ in $\mathbb R^n$. For any face $\tau$ of
$\delta$, let $\sigma_\tau=(\mathrm{cone}_\delta(\tau))^\vee$. We
have $\sigma_\tau^\perp=\mathrm{span}\,\tau$ and
$O_{\sigma_\tau}=\mathrm{Spec}\, k[\mathbb Z^n\cap
\mathrm{span}\,\tau]$. We denote the torus $\mathrm{Spec}\,
k[\mathbb Z^n\cap \mathrm{span}\,\tau]$ by $\mathbb T_\tau$. To fix
notation, we state Lemma 1.2 for this special case as follows:

\begin{lemma} Let $\delta$ be a rational convex polyhedral cone of dimension
$n$ in $\mathbb R^n$, let $\mathrm{cone}^\circ_\delta(\tau)$ be the
image of $\mathrm{cone}_\delta(\tau)$ under the canonical
homomorphism $\mathbb R^n \to \mathbb R^n/\mathrm{span}\,\tau$. It
is a rational convex polyhedral cone in $\mathbb
R^n/\mathrm{span}\,\tau$ of dimension ${\rm dim}(\mathbb
R^n/\mathrm{span}\,\tau)$, and $0$ is a face of
$\mathrm{cone}^\circ_\delta(\tau)$. Let
$$x_\tau:\mathrm{Spec}\,k\to
X(\Sigma(\mathrm{cone}^\circ_\delta(\tau)))=\mathrm{Spec}\,
k[(\mathbb Z^n/\mathbb Z^n\cap \mathrm{span}\,\tau)\cap
\mathrm{cone}^\circ_\delta(\tau)]$$ be the distinguished point in
$X(\Sigma(\mathrm{cone}^\circ_\delta(\tau)))$, let
$$j:{\mathbb T}^n\hookrightarrow X(\Sigma(\delta)),\;j_\tau:{\mathbb T}^{n-\mathrm{dim}\,\tau}
=\mathrm{Spec}\, k[\mathbb Z^n/\mathbb
Z^n\cap\mathrm{span}\,\tau]\hookrightarrow
X(\Sigma(\mathrm{cone}^\circ_\delta(\tau)))$$ be the immersions of
the open dense tori in toric schemes. Denote by $p_\tau:\mathbb
T^n\to \mathbb T_\tau$ the morphism induced by the inclusion
$k[\mathbb Z^n\cap\mathrm{span}\, \tau]\hookrightarrow k[\mathbb
Z^n].$

(i)  If $\mathscr K_\chi$ is not the inverse image under $p_\tau$ of
any Kummer sheaf on $\mathbb T_\tau$, then $(j_{!\ast}(\mathscr
K_\chi[n]))|_{O_{\sigma_\tau}}$ and $(Rj_{\ast}(\mathscr
K_\chi[n]))|_{O_{\sigma_\tau}}$ are acyclic, where $O_{\sigma_\tau}$
is considered as a subscheme of $X(\Sigma(\delta))$.

(ii) If $\mathscr K_\chi=p_\tau^\ast \mathscr K_{\chi_\tau}$ for
some Kummer sheaf $\mathscr K_{\chi_\tau}$ on $\mathbb T_\tau$, then
\begin{eqnarray*}
(j_{!\ast}(\mathscr K_\chi[n]))|_{O_{\sigma_\tau}}&\cong & (\mathscr
K_{\chi_\tau}[{\rm dim}(\tau)])\otimes x_\tau^\ast
({j_\tau}_{!\ast}(\overline {\mathbb Q}_\ell[n-{\rm
dim}(\tau)])),\\
(Rj_{\ast}(\mathscr K_\chi[n]))|_{O_{\sigma_\tau}}&\cong & (\mathscr
K_{\chi_\tau}[{\rm dim}(\tau)])\otimes x_\tau^\ast
({Rj}_{\tau\ast}(\overline {\mathbb Q}_\ell[n-{\rm dim}(\tau)])).
\end{eqnarray*}
\end{lemma}

The morphism $\iota:\mathbb T^n\to\mathbb A'^N$ in Lemma 1.1 is
induced by the $k$-homomorphism
$$\iota^\natural: k[\xi_1,\ldots, \xi_N]\to k[t_1^{\pm 1},\ldots, t_n^{\pm 1}],
\quad \xi_j\mapsto t_1^{w_{1j}}\cdots t_n^{w_{nj}}.$$ Let $\overline
{\mathrm{im}\,\iota}$ be the scheme theoretic image of $\iota$. It
is the closed subscheme $\mathrm {Spec}\, (k[\xi_1,\ldots,
\xi_N]/\mathrm{ker}\, \iota^\natural)$ of $\mathbb A'^N$. We have
$$k[\xi_1,\ldots, \xi_N]/\mathrm{ker}\, \iota^\natural\cong
\mathrm{im}\,\iota^\natural,$$ and it is clear that
$\mathrm{im}\,\iota^\natural$ is spanned by monomials of the form
$t_1^{w_1}\cdots t_n^{w_n}$, where $\mathbf w=\left(\begin{array}{c}
w_1\\ \vdots \\ w_n\end{array} \right)$ lies in the sub-semigroup
$S$ of $\mathbb Z^n$ generated by $\mathbf w_1,\ldots, \mathbf w_N$.
So we have
$$\mathrm{im}\,\iota^\natural\cong k[S].$$
Hence $\overline {\mathrm{im}\,\iota}$ is isomorphic to
$\mathrm{Spec}\, k[S]$, which is an affine toric variety (not
necessarily normal). Let $\delta$ be the convex polyhedral cone in
$\mathbb R^n$ generated by $\mathbf w_1,\ldots, \mathbf w_N$. By
\cite[Proposition 1.3.8]{CLS}, the normalization of $\mathrm{Spec}\,
k[S]$ is the toric variety $X(\Sigma(\delta))=\mathrm {Spec}\,
k[\mathbb Z^n\cap\delta]$. Denote by $g$ the composite
$$X(\Sigma(\delta)) \to \overline{\mathrm{im}\,\iota} \to \mathbb
A'^N.$$ It is induced by the $k$-homomorphism
$$k[\xi_1,\ldots, \xi_N]\to k[\mathbb
Z^n\cap\delta], \quad \xi_j\mapsto \mathbf w_j.$$ Note that $g$ is a
finite morphism. Let $j: \mathbb T^n=\mathrm{Spec}\, k[\mathbb Z^n]
\hookrightarrow X(\Sigma(\delta))$ be the immersion of the open
dense torus. It is induced by the inclusion $k[\mathbb
Z^n\cap\delta] \hookrightarrow k[\mathbb Z^n].$  We have $\iota=g
j$.

\begin{lemma} Suppose that $\chi$
satisfies the non-resonance condition. Then the canonical morphism
$j_!\mathscr K_\chi\to Rj_\ast \mathscr K_{\chi}$ is an isomorphism.
\end{lemma}

\begin{proof} If
$\chi$ satisfies the non-resonance condition, then by Lemma 1.3 (i),
we have $(Rj_\ast\mathscr K_\chi)|_{O_{\sigma_\tau}}=0$ for any
proper face $\tau$ of $\delta$. The union of $O_{\sigma_\tau}$ where
$\tau$ goes over proper faces of $\delta$ is exactly the complement
of the open dense torus in $X(\Sigma(\delta))$. Hence $j_!\mathscr
K_\chi\cong Rj_\ast \mathscr K_\chi$.
\end{proof}

We are now ready to prove Theorem 0.3.

\begin{proof}[Proof of Theorem 0.3]  Since $\iota$ is
quasi-finite and affine, $\iota_!\mathscr K_\chi[n]$ is perverse by
\cite[Corollaire 4.1.3]{BBD}. By Lemma 1.1 and \cite[Th\'eor\`eme
1.3.2.3]{L}, $\mathrm{Hyp}_\psi(\chi)$ is perverse. The perverse
sheaf $\iota_!\mathscr K_\chi[n]$ is mixed of weights $\leq n$. By
Lemma 1.1 and \cite[Th\'eor\`eme 3.3.1]{D},
$\mathrm{Hyp}_\psi(\chi)$ is mixed of weights $\leq n+N$.

Suppose furthermore that $\chi$ satisfies the non-resonance
condition. By Lemma 1.3, the canonical morphism $j_!\mathscr
K_\chi[n]\to Rj_\ast \mathscr K_\chi[n]$ is an isomorphism. We have
$\iota=g j$ and $g$ is a finite morphism. It follows that the
canonical morphism $\iota_!\mathscr K_\chi[n]\to R\iota_\ast
\mathscr K_\chi[n]$ is an isomorphism, and hence $\iota_!\mathscr
K_{\chi}[n]\cong \iota_{!\ast}(\mathscr K_{\chi}[n])$. Suppose
furthermore that $\mathbf {w}_1, \ldots, \mathbf {w}_N$ generate
$\mathbb Z^n$. We will show in a moment that $\iota$ is then an
immersion. So $\iota_{!\ast}(\mathscr K_{\chi}[n])$ is an
irreducible perverse sheaf by \cite[Th\'eor\`eme 4.3.1 (ii)]{BBD}.
By Lemma 1.1 and \cite[Th\'eor\`eme 1.3.2.3]{L},
$\mathrm{Hyp}_\psi(\chi)$ is also an irreducible perverse sheaf. By
\cite[Corollaire 5.4.3]{BBD}, $\iota_{!\ast}(\mathscr K_\chi[n])$ is
a pure perverse sheaf of weight $n$. By Lemma 1.1 and
\cite[Th\'eor\`eme 2.2.1]{KL}, $\mathrm{Hyp}_\psi(\chi)$ is a pure
perverse sheaf of weight $n+N$.

Let's prove $\iota$ is an immersion under the condition that
$\mathbf {w}_1, \ldots, \mathbf {w}_N$ generate $\mathbb Z^n$. There
exists an $(N\times n)$-matrix $B=(v_{ij})$ with integer entries
such that $AB=I_n$. Consider the morphism
$$\iota': \mathbb T^N\to \mathbb T^n,\quad (\xi_1,\ldots, \xi_N)\mapsto
(\xi_1^{v_{11}}\cdots \xi_N^{v_{N1}},\ldots, \xi_1^{v_{1n}}\cdots
\xi_N^{v_{Nn}}).$$ One can verify that
$\iota'\iota=\mathrm{id}_{\mathbb T^n}$. This implies that
$\iota:\mathbb T^n\to \mathbb T^N$ is a closed immersion, and hence
$\iota:\mathbb T^n\to\mathbb A'^N$ is an immersion.
\end{proof}

\begin{theorem} Keep the notation in Lemma 1.3. Let $\delta$ be the convex
polyhedral cone in $\mathbb R^n$ generated by $\mathbf w_1,\ldots,
\mathbf w_N$, let $T$ be the set of proper faces $\tau$ of $\delta$
so that $\mathscr K_\chi\cong p_\tau^\ast\mathscr K_{\chi_\tau}$ for
a Kummer sheaf $\mathscr K_{\chi_\tau}$ on the torus $\mathbb
T_\tau=\mathrm{Spec}\, k[\mathbb Z^n\cap \mathrm{span}\,\tau]$, let
$N_\tau$ be the number of those $\mathbf w_j$ ($j=1,\ldots, N$)
lying in $\tau$, and let $q_\tau:\mathbb A^N\to\mathbb A^{N_\tau}$
be the projection $$(x_1,\ldots, x_N)\mapsto (x_j)_{\mathbf w_j\in
\tau}.$$ Then in the $K$-group $K(D_c^b(\mathbb
A^N,\overline{\mathbb Q}_\ell))$ of $D_c^b(\mathbb
A^N,\overline{\mathbb Q}_\ell)$, we have
\begin{eqnarray*}
\Big[\mathrm{Hyp}_\psi(\chi)\Big]=
\Big[FT_\psi(\iota_{!\ast}(\mathscr K_\chi[n]))\Big]-\sum_{\tau\in
T}\Big[q_\tau^\ast \mathrm{Hyp}_\psi(\chi_\tau)[N-N_\tau]\otimes
x_\tau^\ast(j_{\tau !\ast}\overline{\mathbb Q}_\ell
[n-\mathrm{dim}\,\tau])\Big],
\end{eqnarray*}
where $\mathrm{Hyp}_\psi(\chi_\tau)$ is the GKZ-hypergeometric sheaf
associated to the Kummer sheaf $\chi_\tau$ and the vectors $(\mathbf
w_j)_{\mathbf w_j\in\tau}$ in $\mathbb Z^n\cap\mathrm{span}\,\tau$.
\end{theorem}

\begin{proof} Let
$\kappa_\tau:\mathbb T_\tau\to X(\Sigma(\delta))$ be the immersion
of $O_{\sigma_\tau}=\mathbb T_\tau$ considered as a subscheme of
$X(\Sigma(\delta))$. By Lemma 1.3, in the $K$-group
$K(D_c^b(X(\Sigma(\delta)),\overline{\mathbb Q}_\ell))$ of
$D_c^b(X(\Sigma(\delta)),\overline{\mathbb Q}_\ell)$, we have
\begin{eqnarray*}
\Big[j_{!\ast}(\mathscr K_\chi[n])\Big]=\Big[j_{!}(\mathscr
K_\chi[n])\Big]+\sum_{\tau\in T}\Big[\kappa_{\tau !}(\mathscr
K_{\chi_\tau}[\mathrm{dim}\, \tau])\otimes x_\tau^\ast(j_{\tau
!\ast}(\overline{\mathbb Q}_\ell [n-\mathrm{dim}\,\tau]))\Big].
\end{eqnarray*} Applying $g_!=Rg_\ast$ to the above equality,
we get that in the $K$-group $K(D_c^b(\mathbb A'^N,\overline{\mathbb
Q}_\ell))$ of $D_c^b(\mathbb A'^N,\overline{\mathbb Q}_\ell)$, we
have
\begin{eqnarray*}
\Big[\iota_{!\ast}(\mathscr K_\chi[n])\Big]= \Big[\iota_{!}(\mathscr
K_\chi[n])\Big]+\sum_{\tau\in T}\Big[(g\kappa_\tau)_!(\mathscr
K_{\chi_\tau}[\mathrm{dim}\, \tau])\otimes x_\tau^\ast(j_{\tau
!\ast}(\overline{\mathbb Q}_\ell [n-\mathrm{dim}\,\tau]))\Big].
\end{eqnarray*}
Applying the Deligne-Fourier transformation to this equality, we get
that in the $K$-group $K(D_c^b(\mathbb A^N,\overline{\mathbb
Q}_\ell))$, we have
\begin{eqnarray*}
\Big[FT_\psi(\iota_{!\ast}(\mathscr K_\chi[n]))\Big]=
\Big[FT_\psi(\iota_{!}(\mathscr K_\chi[n]))\Big]+\sum_{\tau\in
T}\Big[FT_\psi\Big((g\kappa_\tau)_!(\mathscr
K_{\chi_\tau}[\mathrm{dim}\, \tau])\Big)\otimes x_\tau^\ast(j_{\tau
!\ast}(\overline{\mathbb Q}_\ell [n-\mathrm{dim}\,\tau]))\Big].
\end{eqnarray*}
Let $i_\tau:\mathbb A'^{N_\tau}\to\mathbb A'^N$ be the inclusion of
the coordinate plane corresponding to those coordinates $\xi_j$ so
that $\mathbf w_j\in \tau$. It is the closed immersion corresponding
to the $k$-homomorphism
$$k[\xi_1,\ldots, \xi_N]\to k[\xi_j]_{\mathbf w_j\in \tau},\quad
\xi_j\mapsto \left\{\begin{array}{cl}\xi_j&\hbox{if } \mathbf w_j\in
\tau,\\
0&\hbox{if }\mathbf w_j\not\in \tau.\end{array}\right.$$ Let
$\iota_\tau:\mathbb T_\tau \to \mathbb A'^{N_\tau}$ be the morphism
corresponding to the $k$-homomorphism
$$k[\xi_j]_{\mathbf w_j\in \tau} \to k[\mathbb Z^n\cap
\mathrm{span}\,\tau],\quad \xi_j\mapsto \mathbf w_j.$$ One can
verify that
$$g\kappa_\tau=i_\tau \iota_\tau.$$ So we have
\begin{eqnarray*}
FT_\psi\Big((g\kappa_\tau)_!(\mathscr K_{\chi_\tau}[\mathrm{dim}\,
\tau])\Big)&\cong& FT_\psi\Big({i_\tau}_\ast {\iota_\tau}_!(\mathscr
K_{\chi_\tau}[\mathrm{dim}\, \tau])\Big)\\
&\cong& q_\tau^\ast\Big(FT_\psi(\iota_{\tau!}(\mathscr
K_{\chi_\tau}[\mathrm{dim}\, \tau]))\Big)[N-N_\tau],
\end{eqnarray*}
where the second isomorphism follows from \cite[Th\'eor\`eme
1.2.2.4]{L}. We thus have
\begin{eqnarray*}
&&\Big[FT_\psi(\iota_{!\ast}(\mathscr K_\chi[n]))\Big]\\&=&
\Big[FT_\psi(\iota_{!}(\mathscr K_\chi[n]))\Big]+\sum_{\tau\in
T}\Big[q_\tau^\ast\Big(FT_\psi(\iota_{\tau!}(\mathscr
K_{\chi_\tau}[\mathrm{dim}\, \tau]))\Big)[N-N_\tau]\otimes
x_\tau^\ast(j_{\tau !\ast}(\overline{\mathbb Q}_\ell
[n-\mathrm{dim}\,\tau]))\Big].
\end{eqnarray*}
By Lemma 1.1 we have $FT_\psi(\iota_{!}(\mathscr
K_\chi[n]))\cong\mathrm{Hyp}_\psi(\chi)$ and
$FT_\psi(\iota_{\tau!}(\mathscr K_{\chi_\tau}[\mathrm{dim}\,
\tau]))\cong\mathrm{Hyp}_\psi(\chi_\tau)$. Our assertion follows.
\end{proof}

We are now ready to prove Theorem 0.4 (i)-(ii).

\begin{proof}[Proof of Theorem 0.4 (i)-(ii)] (i) can be deduced
from results in \cite{DL}. Let
$$f=\sum_{j=1}^N a_jt_1^{w_{1j}}\cdots t_n^{w_{nj}}\in\bar
k[t_1^{\pm 1},\ldots, t_n^{\pm 1}]$$ be a Laurent polynomial
nondegenerate with respect to $\Delta$. Then by \cite[Proposition
0.1]{Fu}, we have
$$\mathrm{dim}\, H_c^i(\mathbb T^n\otimes_k\bar k, \mathscr
K_\chi\otimes f^\ast \mathscr L_\psi)=\left\{\begin{array}{cl}
0&\hbox {if } i\not=n,\\
n!\mathrm{vol}(\Delta)&\hbox{if } i=n.\end{array}\right.$$ As
explained in \cite{Fu}, this result can be deduced from
\cite[Theorem 1.3]{DL}. (\cite[Theorem 1.3]{DL} follows from from
\cite[Theorem 2.7]{DL}, which is deduced from a theorem of
Bernstein-Kushinirenko-Khovanskii over the complex field by using a
comparison theorem and a specialization argument to reduce
characteristic $p$ case to characteristic $0$ case. Using the method
in \cite{K} and \cite[VII 7.3]{SGA5}, one can prove \cite[Theorem
2.7]{DL} directly without passing to characteristic $0$ case.) So we
have
$$\mathrm{dim}\, R\Gamma_c(\mathbb T^n\otimes_k\bar k, \mathscr
K_\chi\otimes f^\ast \mathscr L_\psi)=(-1)^n
n!\mathrm{vol}(\Delta).$$ Denote the $\bar k$-point $(a_1,\ldots,
a_N)$ in $\mathbb A^N$ by $\mathbf a$. We have
$$(\mathrm{Hyp}_\psi(\chi))_{\mathbf a}\cong  R\Gamma_c(\mathbb T^n\otimes_k\bar k, \mathscr
K_\chi\otimes f^\ast \mathscr L_\psi)[n+N].$$ Hence
$$\mathrm{dim}(\mathrm{Hyp}_\psi(\chi))_{\mathbf a}=(-1)^N
n_!\mathrm{vol}(\Delta)$$ for any point geometric point $\mathbf a$
in $V$. So the rank of $\mathrm{Hyp}_\psi(\chi)$ is
$(-1)^Nn_!\mathrm{vol}(\Delta)$.

(ii) Let $K$ be a mixed complex. (Confer \cite[5.1.5]{BBD} for the
definition of mixed complex). For any $i,w\in\mathbb Z$, let
$^p\mathscr H^i(K)_w$ be the weight $w$ sub-quotient of the weight
filtration for the $i$-th perverse cohomolgy sheaf $^p\mathscr
H^i(K)$. Define
$$P(K)=\sum_{i,w\in\mathbb Z}(-1)^i\mathrm{rank}(^p\mathscr H^i(K)_w)T^w.$$
By Theorem 1.5, we have
\begin{eqnarray*}
P(\mathrm{Hyp}_\psi(\chi))&=&P(FT_\psi(\iota_{!\ast}(\mathscr
K_\chi[n])))-\sum_{\tau\in T}P\Big(q_\tau^\ast
\mathrm{Hyp}_\psi(\chi_\tau)[N-N_\tau]\otimes x_\tau^\ast(j_{\tau
!\ast}\overline{\mathbb Q}_\ell [n-\mathrm{dim}\,\tau])\Big).
\end{eqnarray*}
Since $q_\tau$ is smooth of relative dimension $N-N_\tau$, the
functor $q_\tau^\ast [N-N_\tau]\cong Rq_\tau^!
(-(N-N_\tau))[-(N-N_\tau)]$ is exact with respect to the perverse
t-structure. (Confer \cite[4.2.5]{BBD}). Moreover
$\mathrm{Hyp}_\psi(\chi_\tau)$ is perverse by Theorem 0.3. So we
have
$$^p\mathscr H^i \Big(q_\tau^\ast
\mathrm{Hyp}_\psi(\chi_\tau)[N-N_\tau]\otimes x_\tau^\ast(j_{\tau
!\ast}\overline{\mathbb Q}_\ell [n-\mathrm{dim}\,\tau])\Big)\cong
q_\tau^\ast \mathrm{Hyp}_\psi(\chi_\tau)[N-N_\tau]\otimes \mathscr
H^i\Big(x_\tau^\ast(j_{\tau !\ast}\overline{\mathbb Q}_\ell
[n-\mathrm{dim}\,\tau])\Big).$$ Let $\alpha_i$ be the coefficient of
$T^i$ in the polynomial $\alpha(\mathrm{cone}_\delta^\circ(\tau))$.
By \cite[Theorem 6.2]{DL}, $H^i\Big(x_\tau^\ast(j_{\tau
!\ast}\overline{\mathbb Q}_\ell [n-\mathrm{dim}\,\tau])\Big)$ is
pure of weight $i+n-\mathrm{dim}\,\tau$ and $$\mathrm{dim}
H^i\Big(x_\tau^\ast(j_{\tau !\ast}\overline{\mathbb Q}_\ell
[n-\mathrm{dim}\,\tau])\Big)=\alpha_{i+n-\mathrm{dim}\,\tau}.$$ By
\cite[5.1.14]{BBD}, the functor $q_\tau^\ast [N-N_\tau]$ transforms
perverse sheaves pure of weight $w$ to perverse sheaves pure of
weight $w+N-N_\tau$. Let $\mathrm{Hyp}(\chi_\tau)_w$ be the weight
$w$ subquotient of the weight filtration for the mixed perverse
sheaf $\mathrm{Hyp}(\chi_\tau)$. From the above facts, we get
\begin{eqnarray*}
&&P\Big(q_\tau^\ast \mathrm{Hyp}_\psi(\chi_\tau)[N-N_\tau]\otimes
x_\tau^\ast(j_{\tau !\ast}\overline{\mathbb Q}_\ell
[n-\mathrm{dim}\,\tau])\Big)\\
&=& \sum_{i, w}(-1)^i
(-1)^{N-N\tau}\alpha_{i+n-\mathrm{dim}\,\tau}\mathrm{rank}(\mathrm{Hyp}_\psi(\chi_\tau)_w)
T^{w+N-N_\tau+i+n-\mathrm{dim}\,\tau}\\
&=& (-1)^{n-\mathrm{dim}\,\tau+N-N_\tau}  T^{N-N_\tau}\Big(\sum_w
\mathrm{rank}(\mathrm{Hyp}_\psi(\chi_\tau)_w)T^w\Big) \Big(\sum_i
(-1)^{i+n-\mathrm{dim}\,\tau}\alpha_{i+n-\mathrm{dim}\,\tau}
T^{i+n-\mathrm{dim}\,\tau} \Big)\\
&=&(-1)^{n-\mathrm{dim}\,\tau+N-N_\tau}  T^{N-N_\tau}
P(\mathrm{Hyp}_\psi(\chi_\tau))\alpha(\mathrm{cone}^\circ_\delta(\tau)).
\end{eqnarray*}
Here for the last equality, we use the fact
$\alpha(\mathrm{cone}^\circ_\delta(\tau))(-T)=\alpha(\mathrm{cone}^\circ_\delta(\tau))(T)$
since $\alpha(\mathrm{cone}^\circ_\delta(\tau))$ involves only even
powers of $T$. Since $\iota_{!\ast}(\mathscr K_\chi[n])$ is a pure
perverse sheaf of weight $n$, its Deligne-Fourier transform
$FT_\psi(\iota_{!\ast}(\mathscr K_\chi[n]))$ is a pure perverse
sheaf of weight $n+N$. So we have
$$P(FT_\psi(\iota_{!\ast}(\mathscr
K_\chi[n])))=bT^{n+N}$$ for some integer $b$. We thus have
\begin{eqnarray*}
P(\mathrm{Hyp}_\psi(\chi))=bT^{n+N}-\sum_{\tau\in
T}(-1)^{n-\mathrm{dim}\,\tau+N-N_\tau} T^{N-N_\tau}
P(\mathrm{Hyp}_\psi(\chi_\tau))
\alpha(\mathrm{cone}^\circ_\delta(\tau)).
\end{eqnarray*}
Evaluating this equation at $T=1$, and using the fact
\begin{eqnarray*}
P(\mathrm{Hyp}_\psi(\chi))(1)&=&(-1)^N n!\mathrm{vol}(\Delta),
\\
P(\mathrm{Hyp}_\psi(\chi_\tau))(1)&=&(-1)^{N_\tau}
(\mathrm{dim}\,\tau)!\mathrm{vol}(\Delta\cap \tau)
\end{eqnarray*}
which is deduced from (i), we get
$$b=(-1)^N n!\mathrm{vol}(\Delta)+\sum_{\tau\in
T}(-1)^{n-\mathrm{dim}\,\tau+N}
(\mathrm{dim}\,\tau)!\mathrm{vol}(\Delta\cap \tau)
\alpha(\mathrm{cone}^\circ_\delta(\tau))(1),$$ that is, we have
$b=e(\Delta,\chi)$. So we have
\begin{eqnarray*}
P(\mathrm{Hyp}_\psi(\chi))=e(\Delta,\chi) T^{n+N}-\sum_{\tau\in
T}(-1)^{n-\mathrm{dim}\,\tau+N-N_\tau} T^{N-N_\tau}
P(\mathrm{Hyp}_\psi(\chi_\tau))
\alpha(\mathrm{cone}^\circ_\delta(\tau)).
\end{eqnarray*}
By this expression, the definition of $E(\Delta,\chi)$, and
induction on $\mathrm{dim}\,\Delta$, we get
$$P(\mathrm{Hyp}_\psi(\chi))=E(\Delta,\chi).$$
Since $\mathrm{Hyp}_\psi(\chi)$ is a mixed perverse sheaf of weight
$\leq n+N$, the degree of $P(\mathrm{Hyp}_\psi(\chi))$ is at most
$n+N$. Similarly, the degree of $P(\mathrm{Hyp}_\psi(\chi_\tau))$ is
at most $\mathrm{dim}\,\tau+N_\tau$. By the definition of $\alpha$,
we have
$$\mathrm{deg}(\alpha(\mathrm{cone}^\circ_\delta(\tau)))\leq
\mathrm{dim}(\mathrm{cone}^\circ_\delta(\tau))-1=n-\mathrm{dim}\,\tau-1.$$
It follows from the last expression of $P(\mathrm{Hyp}_\psi(\chi))$
that $e_{n+N}=e(\Delta, \chi)$.
\end{proof}

Finally we prove Theorem 0.5.

\begin{proof}[Proof of Theorem 0.5] Theorem 0.5 (i) is Lemma
1.1. Let's prove (ii). In the Abelian category of perverse sheaves,
we have an epimorphism
$$j_!\mathcal K_\chi[n]\twoheadrightarrow j_{!\ast}(\mathcal K_\chi[n]).$$ Let
$\mathcal K$ be the kernel of this epimorphism. We have a
distinguished triangle
$$j_!\mathcal K_\chi[n]\to j_{!\ast}(\mathcal K_\chi[n])\to \mathcal
K[1]\to.$$ On the complement $X(\Sigma(\delta))-\mathbb T^n$ of the
open dense torus, we have $(j_!\mathcal
K_\chi[n])|_{X(\Sigma(\delta))-\mathbb T^n}=0.$ It follows that
$$\mathcal K|_{X(\Sigma(\delta))-\mathbb T^n}\cong
\Big(j_{!\ast}(\mathcal
K_\chi[n])[-1]\Big)|_{X(\Sigma(\delta))-\mathbb T^n}.$$ But
$j_{!\ast}(\mathcal K_\chi[n])$ is a pure complex of weight $n$. So
$\Big(j_{!\ast}(\mathcal
K_\chi[n])[-1]\Big)|_{X(\Sigma(\delta))-\mathbb T^n}$ is mixed of
weight $\leq n-1$. Hence $\mathcal K|_{X(\Sigma(\delta))-\mathbb
T^n}$ is mixed of weight $\leq n-1$. On the other hand, we have
$\mathcal K|_{\mathbb T^n}=0$. So $\mathcal K$ is mixed of weight
$\leq n-1$. Hence $j_{!\ast}(\mathcal K_\chi[n])$ is the weight $n$
subquotient of $j_!\mathcal K_\chi[n]$. Since $\iota=gj$ and $g$ is
a finite morphism, $\iota_{!\ast}(\mathcal K_\chi[n])$ is the weight
$n$ subquotient of $\iota_!\mathcal K_\chi[n]$. Taking the
Deligne-Fourier transformation, we get that
$FT_\psi(\iota_{!\ast}(\mathcal K_\chi[n]))$ is the weight $n+N$
subquotient of $\mathrm{Hyp}_\psi(\chi)\cong FT_\psi(\iota_!\mathcal
K_\chi[n])$ by \cite[Th\'eor\`eme 2.2.1]{KL}. The rank of
$FT_\psi(\iota_{!\ast}(\mathcal K_\chi[n]))$ is $e(\Delta, \chi)$ by
the proof of Theorem 0.4.
\end{proof}

\section{The lisse locus of the GKZ hypergeometric sheaf}

In this section, we prove Theorem 0.4 (iii), that is, for each $i$,
$\mathscr H^i(\mathrm{Hyp}_\psi(\chi))$ is lisse on the open subset
$V$ of $\mathbb A^N$ parametrizing Laurent polynomials nondegenerate
with respect to $\Delta$. Let $[q-1]:\mathbb T^n\to \mathbb T^n$ be
the morphism
$$(t_1,\ldots, t_n)\mapsto (t_1^{q-1}, \ldots, t_n^{q-1}).$$
Then $\mathscr K_\chi$ is a direct factor of
$[q-1]_\ast\overline{\mathbb Q}_\ell$. It follows that
$\mathrm{Hyp}_\psi(\chi)=R\pi_{2!}(\pi_1^\ast\mathscr K_\chi\otimes
F^\ast \mathscr L_\psi)[n+N]$ is a direct factor of
$R\pi_{2!}(\pi_1^\ast[q-1]_\ast \overline{\mathbb Q}_\ell\otimes
F^\ast \mathscr L_\psi)[n+N]$. Fix notation by the following
commutative diagram:
$$\begin{array}{rcrlc} \mathbb T^n&\stackrel{\pi_1}\leftarrow&
\mathbb T^n\times \mathbb
A^N&&\\
{\scriptstyle [q-1]}\downarrow&&{\scriptstyle [q-1]\times\mathrm{id}}\downarrow&\searrow{\scriptstyle \pi_2}&\\
\mathbb T^n&\stackrel{\pi_1}\leftarrow& \mathbb T^n\times \mathbb
A^N&\stackrel{\pi_2}\to&\mathbb A^N.\\
&&{\scriptstyle F}\downarrow &&\\ &&\mathbb A^1&&
\end{array}$$
We have
\begin{eqnarray*}
R\pi_{2!}(\pi_1^\ast[q-1]_\ast \overline{\mathbb Q}_\ell\otimes
F^\ast \mathscr L_\psi)&\cong
&R\pi_{2!}(([q-1]\times\mathrm{id})_\ast\overline{\mathbb
Q}_\ell\otimes F^\ast \mathscr L_\psi)\\
&\cong&R\pi_{2!}([q-1]\times\mathrm{id})_\ast
([q-1]\times\mathrm{id})^\ast F^\ast \mathscr L_\psi\\
&\cong& R\pi_{2 !} (F\circ([q-1]\times\mathrm{id}))^\ast \mathscr
L_\psi.
\end{eqnarray*}
Note that $$F\circ([q-1]\times\mathrm{id})(t_1,\ldots,
t_n,x_1,\ldots, x_N)=\sum_{j=1}^N x_j(t_1^{w_{1j}}\cdots
t_n^{w_{nj}})^{q-1}.$$ So $R\pi_{2!}(\pi_1^\ast[q-1]_\ast
\overline{\mathbb Q}_\ell\otimes F^\ast \mathscr L_\psi)[n+N]$ is
isomorphic to the GKZ hypergeometric sheaf $\mathrm{Hyp}_\psi(1)$
associated to the trivial character $1:
(k^\ast)^n\to\overline{\mathbb Q}_\ell^\ast$ and the vectors
$(q-1)\mathbf w_j$ $(j=1,\ldots,N)$. Moreover, for any rational
point $(a_1,\ldots, a_N)$ of $V$, the Laurent polynomial
$\sum_{j=1}^Na_j (t_1^{w_{1j}}\cdots t_n^{w_{nj}})^{q-1}$ is
non-degenerate with respect to $(q-1)\Delta$. As
$\mathrm{Hyp}_\psi(\chi)$ is a direct factor of
$\mathrm{Hyp}_\psi(1)$, to show $\mathrm{Hyp}_\psi(\chi)$ is lisse
on $V$, it suffices to show $\mathrm{Hyp}_\psi(1)$ is lisse on $V$.
We are thus reduced to the case where $\chi=1$. For simplicity of
notation, we still work with the vectors $\mathbf w_j$ instead of
$(q-1)\mathbf w_j$.

\medskip
Let $\Delta$ be a rational convex polytope of dimension $n$ in
$\mathbb R^n$. Then the set
$$\Sigma(\Delta)=\{(\mathrm{cone}_\Delta(\Gamma))^\vee|\Gamma\prec \Delta\}$$ is a
fan, where for any face $\Gamma$ of $\Delta$,
$\mathrm{cone}_\Delta(\Gamma)$ is the cone in  $\mathbb R^n$
generated by $u'-u$ ($u'\in \Delta$, $u\in \Gamma$), and
$(\mathrm{cone}_\Delta(\Gamma))^\vee$ is its dual. The toric variety
$X(\Sigma(\Delta))$ is proper over $k$, and it has an open covering
$\mathrm{Spec}\, k[\mathbb Z^n\cap \mathrm{cone}_\Delta(\Gamma)]$
$(\Gamma\prec\Delta)$. Let $\bar j:\mathbb T^n\hookrightarrow
X(\Sigma(\Delta))$ be the immersion of the open dense torus, and let
$\bar\pi_2:X(\Sigma(\Delta))\times\mathbb A^N\to\mathbb A^N$ be the
projection. We have a commutative diagram
$$\begin{array}{ccrlc}
&&\mathbb A^1&&\\
&&{\scriptstyle F}\uparrow&&\\
\mathbb T^n&\stackrel{\pi_1}\leftarrow& \mathbb T^n\times \mathbb
A^N&\stackrel{\bar j\times \mathrm{id}}\hookrightarrow &
X(\Sigma(\Delta))\times\mathbb A^N.\\
&&{\scriptstyle \pi_2}\downarrow&\swarrow{\scriptstyle \bar\pi_2}&\\
&&\mathbb A^N&&
\end{array}$$
Since $\bar\pi_2$ is proper, we have
\begin{eqnarray*}
\mathrm{Hyp}_\psi(1)&=&R\pi_{2!}
F^\ast \mathscr L_\psi[n+N]\\
&\cong& R\bar \pi_{2\ast}(\bar j\times \mathrm{id})_!F^\ast \mathscr
L_\psi[n+N].
\end{eqnarray*}
That $\mathrm{Hyp}_\psi(1)$ is lisse on $V$ follows directly from
the following proposition. (Confer \cite[Finitude A 2]{SGA 4 1/2}).

\begin{proposition} Over the open subset $V$ of $\mathbb A^N$, $(\bar j\times
\mathrm{id})_!F^\ast \mathscr L_\psi$ is universally locally acyclic
relative to $\bar \pi_2$.
\end{proposition}

\begin{proof} As in \cite[3.7.3]{D}, we prove that over $V$, the pair
$\Big((\bar j\times \mathrm{id})_!F^\ast \mathscr L_\psi, \bar
\pi_2\Big)$ is locally constant, that is, locally with respect to
the \'etale topology, it is isomorphic to the base change $(p_1^\ast
\mathscr L, p_2)$ to $\mathbb A^N$ of a pair $(\mathscr L,X)$, where
$X$ is a $k$-scheme, $\mathscr L$ is a $\overline{\mathbb
Q}_\ell$-sheaf on $X$, $p_1:X\times_k\mathbb A^N\to X$ and
$p_2:X\times_k\mathbb A^N\to\mathbb A^N$ are projections. Our
assertion then follows from \cite[Finitude 2.16]{SGA 4 1/2}.

Let $\Gamma$ be a face of $\Delta$, and let
$\sigma_\Gamma=(\mathrm{cone}_\Delta(\Gamma))^\vee$ be the cone in
$\Sigma(\Delta)$ corresponding to $\Gamma$. We have
$$\sigma_\Gamma^\perp=\mathrm{span}(\Gamma-\Gamma),$$ where $\Gamma-\Gamma=
\{u'-u|u,u'\in\Gamma\}$.  We have a closed immersion
$$O_{\sigma_\Gamma}=\mathrm{Spec}\, k[\mathbb Z^n\cap\mathrm{span}(\Gamma-\Gamma)]\to
U_{\sigma_\Gamma}=\mathrm{Spec}\, k[\mathbb Z^n\cap
\mathrm{cone}_\Delta(\Gamma)]$$ induced by the $k$-epimorphism
$$k[\mathbb Z^n\cap\mathrm{cone}_\Delta(\Gamma)]\to k[\mathbb Z^n\cap
\mathrm{span}(\Gamma-\Gamma)], \quad u\mapsto
\left\{\begin{array}{cl} u&\hbox{if } u\in
\mathrm{span}(\Gamma-\Gamma)\\
0&\hbox{if
}u\not\in\mathrm{span}(\Gamma-\Gamma)\end{array}\right.\hbox{for any
} u\in \mathbb Z^n\cap \mathrm{cone}_\Delta(\Gamma).$$ Regard
$O_{\sigma_\Gamma}$ as a subscheme of $X(\Sigma)$. We have
$X(\Sigma(\Delta))=\coprod_{\Gamma\prec\Delta} O_{\sigma_{\Gamma}}$.

First consider the case where $\Gamma$ is a face of $\Delta$
containing $0$. Then we have $\mathbf
w_j\in\mathrm{cone}_\Delta(\Gamma)$ for all $j=1,\ldots, N$. So the
polynomial $\sum_{j=1}^N x_jt_1^{w_{1j}}\cdots t_n^{w_{nj}}$ can be
regarded as an element in
$$\Gamma(U_{\sigma_\Gamma}\times\mathbb A^N,\mathscr
O_{U_{\sigma_\Gamma}\times\mathbb A^N})\cong k[\mathbb Z^n\cap
\mathrm{cone}_{\Delta}(\Gamma)][x_1,\ldots,x_N].$$ Hence the
morphism $F:\mathbb T^n\times \mathbb A^N\to\mathbb A^1$ can be
extended to a morphism $\overline F:U_{\sigma_\Gamma}\times\mathbb
A^N\to \mathbb A^1$. The sheaf $\overline F^\ast \mathscr L_\psi$ is
locally constant on $U_{\sigma_\Gamma}\times \mathbb A^N$. Let
$j:\mathbb T^n\to U_{\sigma_\Gamma}$ be the immersion of the open
dense torus. Then the restriction of the pair $\Big((\bar j\times
\mathrm{id})_!F^\ast \mathscr L_\psi, \bar \pi_2\Big)$ to
$U_{\sigma_\Gamma}\times\mathbb A^N$ is isomorphic to the pair
$$\Big((j\times\mathrm{id})_!(j\times\mathrm{id})^\ast
\overline F^\ast \mathscr L_\psi, U_{\sigma_\Gamma}\times\mathbb
A^N\to\mathbb A^N\Big),$$ and locally with respect to the \'etale
topology, the second pair is isomorphic to the base change to
$\mathbb A^N$ of the pair $(j_!\overline{\mathbb Q}_\ell,
U_{\sigma_\Gamma})$.

Now suppose $\Gamma$ is a face of $\Delta$ not containing $0$, and
let $(P, \mathbf a):\mathrm{Spec}\,\bar k\to X(\Sigma(\Delta))\times
V$ be a $\bar k$-point of $X(\Sigma(\Delta))\times V$, where
$P:\mathrm{Spec}\,\bar k\to X(\Sigma(\Delta))$ is a $\bar k$-point
of $X(\Sigma(\Delta))$ lying in $O_{\sigma_\Gamma}$, and $\mathbf a$
is a $\bar k$-point of $\mathbb A^N$ with coordinate $(a_1,\ldots,
a_N)$ lying in $V$. Let's prove the pair $\Big(\bar \pi_2,(\bar
j\times \mathrm{id})_!F^\ast \mathscr L_\psi\Big)$ is locally
constant near the point $(P,\mathbf a)$. As
$X(\Sigma(\Delta))=\coprod_{\Gamma\prec\Delta} O_{\sigma_{\Gamma}}$,
this will finish the proof of our assertion.

We first make some simplification. Choose a codimension 1 face
$\bar\Gamma$ of $\Delta$ not containing the origin but containing
$\Gamma$, and choose relatively prime integers $d_1,\ldots, d_n$ so
that the restriction to $\Delta$ of the linear function
$$\phi(v_1,\ldots, v_n)=d_1v_1+\cdots +d_nv_n$$ takes its minimum
$d$ exactly on the codimension one face $\bar\Gamma$. By our
assumption, $p$ does not divide $d$, and we have
$$\phi|_{\Delta}\geq d,\quad \Gamma\subset \{\mathbf v\in
\Delta|\phi(\mathbf v)=d\}.$$ Since $d_1,\ldots, d_n$ are relatively
prime, the abelian group $\mathbb Z^n/\mathbb Z(d_1,\ldots, d_n)$ is
torsion free. Indeed, there exist integers $a_1,\ldots,
a_n\in\mathbb Z$ such that
$$\sum_{i=1}^n a_i d_i=1.$$ If $(x_1,\ldots,
x_n)\in\mathbb Z^n$ and $k(x_1,\ldots, x_n)=m(d_1,\ldots, d_n)$ for
some integers $k\not=0$ and $m$, then $k$ divides $md_1,\ldots,
md_n$ and hence $k$ divides $m=\sum_{i=1}^n ma_id_i$. It follows
that $(x_1,\ldots, x_n)=\frac{m}{k}(d_1,\ldots, d_n)$ lies in the
subgroup of $\mathbb Z^n$ generated by $(d_1,\ldots, d_n)$. So
$\mathbb Z^n/\mathbb Z(d_1,\ldots, d_n)$ is torsion free, and hence
free. It follows that $\mathbb Z(d_1,\ldots, d_n)$ is a direct
factor of $\mathbb Z^n$. Hence $\phi$ is part of a basis of
$\mathrm{Hom}(\mathbb Z^n, \mathbb Z)$. Note that
$\phi|_{\mathrm{span}(\Gamma-\Gamma)}=0$. Choose a basis
$\{\phi_1,\ldots, \phi_n\}$ of $\mathrm{Hom}(\mathbb Z^n,\mathbb Z)$
so that $\phi_n=\phi$ and that $\{\phi_{m+1},\ldots, \phi_n\}$ form
a basis of $\mathrm{span}(\Gamma-\Gamma)^\perp$
$(m=\mathrm{dim}\,\Gamma)$. Let $\{e_1,\ldots, e_n\}$ be the dual
basis of $\{\phi_1,\ldots, \phi_n\}$. We have
$$\mathrm{span}(\Gamma-\Gamma)=\mathrm{span}\{e_1,\ldots, e_m\}.$$
Expand elements in $\mathbb Z^n$ with respect to this basis and
denote $w_1e_1+\cdots +w_ne_n$ by $\mathbf w=(w_1,\ldots, w_n)$. For
any $\mathbf w\in \Delta$ (resp. $\mathbf w\in \Gamma$), we have
$w_n\geq d$ (resp. $w_n=d$). Moreover, for any $m+1\leq i\leq n$,
elements in $\Gamma$ have a common $i$-th coordinate.

Let $\mathrm{cone}^\circ_\Delta(\Gamma)$ be the image of
$\mathrm{cone}_\Delta(\Gamma)$ under the projection $\mathbb R^n\to
\mathrm{span}\{e_{m+1}, \ldots, e_n\}$. We have
\begin{eqnarray*}
\mathrm{cone}_\Delta(\Gamma)&=&\mathrm{span}(\Gamma-\Gamma)\oplus
\mathrm{cone}^\circ_\Delta(\Gamma),\\
k[\mathbb Z^n\cap \mathrm{cone}_\Delta(\Gamma)]&\cong& k[\mathbb
Z^{n-m}\cap \mathrm{cone}^\circ_\Delta(\Gamma)][t_1^{\pm 1},\ldots,
t_m^{\pm 1}].
\end{eqnarray*}
Since the $\bar k$-point $P$ of $X(\Sigma(\Delta))$ lies in
$O_{\sigma_\Gamma}$, $P$ corresponds to a $k$-homomorphism
$$P^\natural: k[\mathbb Z^n\cap \mathrm{cone}_\Delta(\Gamma)]\to \bar k$$
such that for any $u\in\mathbb Z^n\cap
\mathrm{cone}_\Delta(\Gamma)$, we have $P^\natural(u)=0$ if $u\not
\in \mathrm{span}(\Gamma-\Gamma)$, and $P^\natural(u)$ is a nonzero
element in $k^\ast$ if $u\in\mathrm{span}(\Gamma-\Gamma)$. Fix
$j_0\in \{1,\ldots,N\}$ so that $\mathbf w_{j_0}=(w_{1j_0},\ldots,
w_{nj_0})$ lies in $\Gamma$. Write
$$t_1^{-w_{1j_0}}\cdots t_n^{-w_{nj_0}}F=\sum_{j=1}^N
x_jt_1^{i_{1j}}\cdots t_n^{i_{nj}},$$ where $(i_{1j},\ldots,
i_{nj})=(w_{1j},\ldots, w_{nj})-(w_{1j_0},\ldots, w_{nj_0})$. This
is an element in $k[\mathbb Z^n\cap
\mathrm{cone}_\Delta(\Gamma)][x_1, \ldots, x_N]$. Set
$F_\Gamma=\sum_{\mathbf w_j\in\Gamma} x_j t_1^{w_{1j}} \cdots
t_n^{w_{nj}}$. Then since elements in $\Gamma$ have a common $i$-th
coordinate for any $m+1\leq i\leq n$, the polynomial
$t_1^{-w_{1j_0}}\cdots t_n^{-w_{nj_0}}F_\Gamma$ is a sum of
monomials involving only $t_1,\ldots, t_m$, that is,
$$t_1^{-w_{1j_0}}\cdots t_n^{-w_{nj_0}}F_\Gamma=\sum_{\mathbf w_j\in\Gamma} x_j t_1^{i_{1j}}
\cdots t_m^{i_{mj}}.$$ Since $f=\sum_{j=1}^N a_jt_1^{w_{1j}}\cdots
t_n^{w_{nj}}$ is nondegenerate with respect to $\Delta$,
$$\frac{\partial f_\Gamma}{\partial t_1}=\cdots=\frac{\partial f_\Gamma}{\partial
t_n}=0$$ has no solution in $(k^\ast)^n$. This implies that
$$t_1^{-w_{1j_0}}\cdots t_n^{-w_{nj_0}}f_\Gamma=
\frac{\partial}{\partial t_1}(t_1^{-w_{1j_0}}\cdots
t_n^{-w_{nj_0}}f_\Gamma)=\cdots=\frac{\partial}{\partial
t_m}(t_1^{-w_{1j_0}}\cdots t_n^{-w_{nj_0}}f_\Gamma)=0$$ has no
solution in $(k^\ast)^n$. In particular, not all
$$(t_1^{-w_{1j_0}}\cdots t_n^{-w_{nj_0}}F_\Gamma)(P,\mathbf a),
\Big(\frac{\partial}{\partial t_1}(t_1^{-w_{1j_0}}\cdots
t_n^{-w_{nj_0}}F_\Gamma)\Big)(P,\mathbf a),\ldots,
\Big(\frac{\partial}{\partial t_m}(t_1^{-w_{1j_0}}\cdots
t_n^{-w_{nj_0}}F_\Gamma)\Big)(P,\mathbf a)$$ are zero.

\emph{Case 1.} $(t_1^{-w_{1j_0}}\cdots
t_n^{-w_{nj_0}}F_\Gamma)(P,\mathbf a)=0$.

Then there exists $i\in\{1,\ldots, m\}$ such that
$\Big(\frac{\partial}{\partial t_i}(t_1^{-w_{1j_0}}\cdots
t_n^{-w_{nj_0}}F_\Gamma)\Big)(P,\mathbf a)\not=0$. Without loss of
generality, suppose $\Big(\frac{\partial}{\partial
t_m}(t_1^{-w_{1j_0}}\cdots t_n^{-w_{nj_0}}F_\Gamma)\Big)(P,\mathbf
a)\not=0$. Consider the $k$-homomorphism
\begin{eqnarray*}
\varphi^\natural: k[\mathbb Z^{n-m}\cap
\mathrm{cone}^\circ_\Delta(\Gamma)][t_1^{\pm 1},\ldots, t_{m-1}^{\pm
1},t_m, x_1, \ldots, x_N]&\to& k[\mathbb Z^n\cap
\mathrm{cone}_\Delta(\Gamma)][x_1, \ldots, x_N]\\
&=& k[\mathbb Z^{n-m}\cap
\mathrm{cone}^\circ_\Delta(\Gamma)][t_1^{\pm 1},\ldots, t_{m}^{\pm
1}, x_1, \ldots, x_N]
\end{eqnarray*}
defined by
\begin{eqnarray*}
\qquad t_m&\mapsto& (t_1^{-w_{1j_0}}\cdots t_n^{-w_{nj_0}}F) t_1^{w_{1j_0}}\cdots t_m^{w_{mj_0}}, \\
g&\mapsto& g \quad \hbox {if } g\in k[\mathbb Z^{n-m}\cap
\mathrm{cone}^\circ_\Delta(\Gamma)][t_1^{\pm 1},\ldots,t_{m-1}^{\pm
1}, x_1, \ldots, x_N].
\end{eqnarray*}
It induces a $k$-morphism
$$\varphi: U_{\sigma_\Gamma}\times\mathbb A^N\to\Big(\mathrm{Spec}\,k[\mathbb Z^{n-m}\cap
\mathrm{cone}^\circ_\Delta(\Gamma)][t_1^{\pm 1},\ldots, t_{m-1}^{\pm
1},t_m]\Big)\times \mathbb A^N.$$ Note that if $\mathbf
w_j\in\Delta\backslash \Gamma$, then we have $\mathbf w_j-\mathbf
w_{j_0}\in\mathrm{cone}_\Delta(\Gamma)\backslash
\mathrm{span}(\Gamma-\Gamma)$ and hence $P^\natural (\mathbf
w_j-\mathbf w_{j_0})=0$. It follows that
\begin{eqnarray*}
(t_1^{w_{1j}-w_{1j_0}}\cdots t_n^{w_{nj}-w_{nj_0}})(P,\mathbf a)
&=& 0\\
\Big({t_{m}}\frac{\partial}{\partial
t_{m}}(t_1^{w_{1j}-w_{1j_0}}\cdots
t_n^{w_{nj}-w_{nj_0}})\Big)(P,\mathbf a)&=& \Big((w_{mj}-w_{mj_0})
t_1^{w_{1j}-w_{1j_0}}\cdots t_n^{w_{nj}-w_{nj_0}}\Big)(P,\mathbf
a)=0.
\end{eqnarray*}
So we have
\begin{eqnarray*}
(t_1^{-w_{1j_0}}\cdots t_n^{-w_{nj_0}}F)(P,\mathbf a) &=&
(t_1^{-w_{1j_0}}\cdots t_n^{-w_{nj_0}}F_\Gamma)(P,\mathbf a)\\
&=&0. \\
\Big({t_m}\frac{\partial}{\partial t_m}(t_1^{-w_{1j_0}}\cdots
t_n^{-w_{nj_0}}F)\Big)(P,\mathbf a) &=&
\Big({t_m}\frac{\partial}{\partial
t_m}(t_1^{-w_{1j_0}}\cdots t_n^{-w_{nj_0}}F_\Gamma)\Big)(P,\mathbf a)\\
&\not=& 0.
\end{eqnarray*}
On the other hand, $(w_{1j_0}, \cdots, w_{mj_0},0,\ldots, 0)$ lies
in $\mathrm{span}(\Gamma-\Gamma)$, we have
$$(t_1^{w_{1j_0}}\cdots t_m^{w_{mj_0}})(P,\mathbf a)\not=0.$$ It follows
that
\begin{eqnarray*}
&&\Big({t_m}\frac{\partial}{\partial t_m}\Big((t_1^{-w_{1j_0}}\cdots
t_n^{-w_{nj_0}}F) t_1^{w_{1j_0}}\cdots
t_m^{w_{mj_0}}\Big)\Big)(P,\mathbf a)\\ &=&
\Big({t_m}\frac{\partial}{\partial t_m}(t_1^{-w_{1j_0}}\cdots
t_n^{-w_{nj_0}}F)\Big)(P,\mathbf a) \cdot  (t_1^{w_{1j_0}}\cdots
t_m^{w_{mj_0}})(P,\mathbf a)\\
&&\qquad \qquad + (t_1^{-w_{1j_0}}\cdots t_n^{-w_{nj_0}}F)(P,\mathbf
a)\Big({t_m}\frac{\partial}{\partial t_m}( t_1^{w_{1j_0}}\cdots
t_m^{w_{mj_0}})\Big)(P,\mathbf a)\\
&\not=&0.
\end{eqnarray*}
By the Jacobian criterion, $\varphi$ is etale at the $\bar k$-point
$(P,\mathbf a)$ of $U_{\sigma_\Gamma}\times\mathbb A^N$. Let
$$j:\mathrm{Spec}\,k[\mathbb Z^{n-m}][t_1^{\pm 1},\ldots, t_{m-1}^{\pm 1}, t_m]
\hookrightarrow \mathrm{Spec}\, k[\mathbb Z^{n-m}\cap
\mathrm{cone}_\Delta^\circ(\Gamma)][t_1^{\pm 1},\ldots, t_{m-1}^{\pm
1},t_m]$$ be the canonical open immersion, and let $F_0$ be the
morphism
$$\mathrm{Spec}\,k[\mathbb Z^{n-m}][t_1^{\pm 1},\ldots, t_{m-1}^{\pm 1}, t_m]
\to \mathbb A^1, \quad (t_1,\ldots, t_n)\to t_m t_{m+1}^{w_{m+1,
j_0}}\cdots t_n^{w_{nj_0}}.$$ Through the etale morphism $\varphi$,
locally near $(P,\mathbf a)$ with respect to the etale topology, the
pair $\Big((\bar j\times \mathrm{id})_!F^\ast \mathscr L_\psi, \bar
\pi_2\Big)$ is isomorphic to the base change to $\mathbb A^N$ of the
pair $(j_!F_0^\ast \mathscr L_\psi, \mathrm{Spec}\,k[\mathbb
Z^{n-m}\cap \mathrm{cone}_\Delta^\circ(\Gamma)][t_1^{\pm 1},\ldots,
t_{m-1}^{\pm 1},t_m])$.

\medskip
\emph{Case 2.} $(t_1^{-w_{1j_0}}\cdots
t_n^{-w_{nj_0}}F_\Gamma)(P,\mathbf a)\not=0$.

Recall that $w_{nj}=d$ for all $\mathbf w_j\in \Gamma$. Since $p$
does not divide $d$, the canonical morphism
{\small\begin{eqnarray*}
\psi: \mathrm{Spec}\, k[\mathbb Z^n\cap
\mathrm{cone}_\Delta(\Gamma)][x_1,\ldots, x_n][T, T^{-1}]/(T^d-
t_1^{-w_{1j_0}}\cdots t_n^{-w_{nj_0}}F) &\to& \mathrm{Spec}\,
k[\mathbb Z^n\cap \mathrm{cone}_\Delta(\Gamma)][x_1,\ldots,
x_n]\\&=&U_{\sigma_\Gamma}\times \mathbb A^N
\end{eqnarray*}}
is etale over the $\bar k$-point $(P,\mathbf a)$. Consider the
$k$-homomorphism
$$\theta^\natural: k[\mathbb Z^n\cap
\mathrm{cone}_\Delta(\Gamma)][x_1,\ldots, x_n]\to k[\mathbb Z^n\cap
\mathrm{cone}_\Delta(\Gamma)][x_1,\ldots, x_n][T, T^{-1}]/(T^d-
t_1^{-w_{1j_0}}\cdots t_n^{-w_{nj_0}}F)$$ defined by
\begin{eqnarray*}
t_1^{w_1}\cdots t_n^{w_n}&\mapsto& t_1^{w_1}\cdots t_n^{w_n} T^{w_n}
\quad \hbox{for any } \mathbf w=(w_1,\ldots,w_n)\in \mathbb
Z^n\cap \mathrm{cone}_\Delta(\Gamma)\\
g&\mapsto& g \quad\hbox {for any } g\in k[x_1,\ldots, x_n],
\end{eqnarray*}
It induces a $k$-morphism
{\small\begin{eqnarray*}
\theta:
\mathrm{Spec}\, k[\mathbb Z^n\cap
\mathrm{cone}_\Delta(\Gamma)][x_1,\ldots, x_n][T, T^{-1}]/(T^d-
t_1^{-w_{1j_0}}\cdots t_n^{-w_{nj_0}}F)&\to& \mathrm{Spec}\,
k[\mathbb Z^n\cap \mathrm{cone}_\Delta(\Gamma)][x_1,\ldots,
x_n]\\
&=&U_{\sigma_\Gamma}\times \mathbb A^N.
\end{eqnarray*}}
We claim that $\theta$ is etale over $(P,\mathbf a)$. Indeed, we
have a commutative diagram
{\footnotesize$$\begin{array}{ccc} \small
k[\mathbb Z^n\cap \mathrm{cone}_\Delta(\Gamma)][x_1,\ldots, x_n][T,
T^{-1}]/(T^d- \widetilde F)&\stackrel{\cong}\to& k[\mathbb Z^n\cap
\mathrm{cone}_\Delta(\Gamma)][x_1,\ldots, x_n][T, T^{-1}]/(T^d-
t_1^{-w_{1j_0}}\cdots t_n^{-w_{nj_0}}F)\\
\uparrow &&\uparrow \\
k[\mathbb Z^n\cap \mathrm{cone}_\Delta(\Gamma)][x_1,\ldots, x_n][T,
T^{-1}]&\stackrel{\cong}\to&k[\mathbb Z^n\cap
\mathrm{cone}_\Delta(\Gamma)][x_1,\ldots, x_n][T, T^{-1}]\\
\uparrow&&\\
k[\mathbb Z^n\cap \mathrm{cone}_\Delta(\Gamma)][x_1,\ldots, x_n],
\end{array}$$}
where the vertical arrows are canonical homomorphisms, and
horizontal arrows are isomorphisms defined by
\begin{eqnarray*}
t_1^{w_1}\cdots t_n^{w_n}&\mapsto& t_1^{w_1}\cdots t_n^{w_n} T^{w_n}
\quad \hbox{for any } \mathbf w=(w_1,\ldots,w_n)\in \mathbb
Z^n\cap \mathrm{cone}_\Delta(\Gamma),\\
g&\mapsto& g \quad\hbox {for any } g\in k[x_1,\ldots, x_n, T,
T^{-1}],
\end{eqnarray*}
and $\widetilde F$ is the preimage of $t_1^{-w_{1j_0}}\cdots
t_n^{-w_{nj_0}}F$ with respect to the second horizontal arrow, that
is,
\begin{eqnarray*}
\widetilde F &=&t_1^{-w_{1j_0}}\cdots t_n^{-w_{nj_0}}
T^{w_{nj_0}}\Big(\sum_{j=1}^N x_j t_1^{w_{1j}}\cdots
t_n^{w_{nj}}T^{-w_{nj}}\Big)\\
&=& \sum_{j=1}^N x_j t_1^{w_{1j}-w_{1j_0}}\cdots
t_n^{w_{nj}-w_{nj_0}}T^{w_{nj_0}-w_{nj}}.
\end{eqnarray*}
To prove $\theta$ is etale over $(P,\mathbf a)$, it suffices to show
the canonical morphism
$$\theta':\mathrm{Spec}\,
k[\mathbb Z^n\cap \mathrm{cone}_\Delta(\Gamma)][x_1,\ldots, x_n][T,
T^{-1}]/(T^d- \widetilde F) \to \mathrm{Spec}\, k[\mathbb Z^n\cap
\mathrm{cone}_\Delta(\Gamma)][x_1,\ldots, x_n]$$ is etale over
$(P,\mathbf a)$. Indeed, if $\mathbf w_j\in \Delta\backslash\Gamma$,
we have $\mathbf w_j-\mathbf
w_{j_0}\in\mathrm{cone}_\Delta(\Gamma)\backslash\mathrm{span}(\Gamma-\Gamma)$
and hence
$$(t_1^{w_{1j}-w_{1j_0}}\cdots
t_n^{w_{nj}-w_{nj_0}})(P,\mathbf a)=0.$$ If $\mathbf w_j\in\Gamma$,
then $w_{nj}=w_{nj_0}=d$ and hence ${w_{nj_0}-w_{nj}}=0.$ In any
case, for any $j=1,\ldots, N$, we have
\begin{eqnarray*}
&&\Big(\frac{\partial}{\partial T}(x_jt_1^{w_{1j}-w_{1j_0}}\cdots
t_n^{w_{nj}-w_{nj_0}}T^{w_{nj_0}-w_{nj}})\Big)(P,\mathbf a)\\&=&
\Big(({w_{nj_0}-w_{nj}})x_j t_1^{w_{1j}-w_{1j_0}}\cdots
t_n^{w_{nj}-w_{nj_0}}T^{w_{nj_0}-w_{nj}-1}\Big)(P,\mathbf a)\\
&=&0.
\end{eqnarray*}
It follows that
\begin{eqnarray*}
&&\Big(\frac{\partial}{\partial T}(T^d-\widetilde F)\Big)(P,\mathbf a)\\
&=&\Big(\frac{\partial}{\partial T}\Big(T^d-\sum_{j=1}^N x_j
t_1^{w_{1j}-w_{1j_0}}\cdots
t_n^{w_{nj}-w_{nj_0}}T^{w_{nj_0}-w_{nj}}\Big)\Big)(P,\mathbf a)\\
&=& dT^{d-1}.
\end{eqnarray*}
By the Jacobian criterion, $\theta'$ is etale over $(P,\mathbf a)$.
This proves our claim. Let $$j:\mathbb T^n=\mathrm{Spec}\,k[\mathbb
Z^n] \hookrightarrow \mathrm{Spec}\, k[\mathbb Z^n\cap
\mathrm{cone}_\Delta(\Gamma)]$$ be the canonical open immersion, and
let $F_0$ be the morphism
$$\mathbb T^n\to \mathbb A^1, \quad (t_1,\ldots, t_n)\to t_1^{w_{1j_0}}
\cdots t_n^{w_{nj_0}}.$$ Through the etale morphisms $\psi$ and
$\theta$, locally near $(P,\mathbf a)$ with respect to the etale
topology, the pair $\Big((\bar j\times \mathrm{id})_!F^\ast \mathscr
L_\psi, \bar \pi_2\Big)$ is isomorphic to the base change to
$\mathbb A^N$ of the pair $(j_!F_0^\ast \mathscr L_\psi,
\mathrm{Spec}\,k[\mathbb Z^n\cap \mathrm{cone}_\Delta(\Gamma)])$.
\end{proof}

\section{Proof of Theorem 0.8}

Suppose $A$ satisfies the nonconfluence condition, that is, there
exist integers $c_1,\ldots, c_n\in\mathbb Z$ such that
$$\sum_{i=1}^n c_iw_{ij}=1 \quad (j=1,\ldots, N).$$ Note that
$c_1,\ldots, c_n\in\mathbb Z$ are necessarily relatively prime. This
implies that the quotient group $\mathbb Z^n/\{m(c_1,\ldots,
c_n)|m\in\mathbb Z\}$ is free. Therefore $\{m(c_1,\ldots,
c_n)|m\in\mathbb Z\}$ is a direct factor of $\mathbb Z^n$, and
$(c_1,\ldots, c_n)$ can be extended to a basis of $\mathbb Z^n$. So
we can find an $(n\times n)$-matrix $C=(c_{ij})$ with integer
entries so that its row vectors form a basis of $\mathbb Z^n$ and
that its first row is exactly $(c_1,\ldots, c_n)$. The matrix $C$ is
an invertible matrix in $\mathrm{GL}(n,\mathbb Z)$, and
$$\sum_{i=1}^n c_{1i} w_{ij}=1 \quad (j=1,\ldots,N).$$ Set
$$w_{kj}'=\sum_{i=1}^n c_{ki}w_{ij} \quad (k=2,\ldots, n,\;
j=1,\ldots, N).$$ We first give an heuristic argument explaining the
main idea of the proof of Theorem 0.8. Make the change of variables
$$t_1=s_1^{c_{11}}\cdots s_n^{c_{n1}},\ldots,
t_n=s_1^{c_{1n}}\cdots s_n^{c_{nn}},$$ and let
$$\chi_i'=\chi_1^{c_{i1}}\cdots \chi_n^{c_{in}}\quad (i=1,\ldots,
n).$$ Then we can evaluate the GKZ hypergeometric sum as follows:
\begin{eqnarray*}
&&\mathrm{Hyp}_\psi(x_1,\ldots, x_N;\chi_1,\ldots,\chi_n) \\
&=& \sum_{t_1,\ldots, t_n\in
k^\ast}\chi_1(t_1)\cdots\chi_n(t_n)\psi\Big(\sum_{j=1}^N x_j
t_1^{w_{1j}}\cdots t_n^{w_{nj}}\Big)\\
&=& \sum_{s_1,\ldots, s_n\in k^\ast}\chi_1(s_1^{c_{11}}\cdots
s_n^{c_{n1}})\cdots\chi_n(s_1^{c_{1n}}\cdots
s_n^{c_{nn}})\psi\Big(\sum_{j=1}^Nx_j s_1^{\sum_{i=1}^n
c_{1i}w_{ij}}
\cdots  s_n^{\sum_{i=1}^n c_{ni}w_{ij}}\Big)\\
&=& \sum_{s_1,\ldots, s_n\in k^\ast}\chi'_1(s_1)\cdots
\chi'_n(s_n)\psi\Big(\sum_{j=1}^Nx_js_1s_2^{w'_{2j}}\cdots
s_n^{w'_{nj}} \Big)\\ &=& \sum_{s_2,\ldots, s_n\in
k^\ast}\chi'_2(s_2)\cdots \chi'_n(s_n)\sum_{s_1\in\mathbb
F_q^\ast}\chi_1'(s_1) \psi\Big(s_1\sum_{j=1}^Nx_j
s_2^{w'_{2j}}\cdots
s_n^{w'_{nj}}\Big)\\
&=&\Big(\sum_{s_2,\ldots, s_n\in k^\ast}\chi'_2(s_2)\cdots
\chi'_n(s_n) \chi_1^{\prime
-1}\Big(\sum_{j=1}^Nx_js_2^{w'_{2j}}\cdots s_n^{w'_{nj}} \Big)\Big)
g(\chi'_1,\psi),
\end{eqnarray*}
where $g(\chi'_1,\psi)=\sum_{t\in k^\ast} \chi_1'(t)\psi(t)$ is the
Gauss sum, and the last equality holds if $\chi_1'$ is nontrivial.
This shows that under this condition, the GKZ hypergeometric sum is
a product of the Gauss sum and a sum involving only multiplicative
characters. We now give the proof of Theorem 0.8.

\medskip
\noindent {\it Proof of Theorem 0.8.} Keep the notation above. The
morphism
\begin{eqnarray*}
H:\mathbb T^n&\to&\mathbb T^n, \\
H(s_1,\ldots, s_n)&=& (s_1^{c_{11}}\cdots s_n^{c_{n1}},\ldots,
s_1^{c_{1n}}\cdots s_n^{c_{nn}})
\end{eqnarray*}
is an isomorphism. Fix notations by the following diagram
$$\begin{array}{ccrcc}
&&\mathbb T^n\times\mathbb A^N&&\\
&&{\scriptstyle H\times {\mathrm{id}_{\mathbb A^N}}}\downarrow&&\\
\mathbb T^n&\stackrel{\pi_1}\leftarrow&\mathbb T^n\times\mathbb A^N&
\stackrel{\pi_2}\to& \mathbb A^N,\\
&&{\scriptstyle F}\downarrow&& \\
&&\mathbb A^1&&
\end{array}$$
where $\pi_1$ and $\pi_2$ are projections, and $F$ is the morphism
defined by
$$F(t_1,\ldots, t_n, x_1,\ldots,
x_N)=\sum_{j=1}^Nx_jt_1^{w_{1j}}\cdots t_n^{w_{nj}}.$$ We have
$\pi_2(H\times\mathrm{id}_{\mathbb A^N})=\pi_2$,
$\pi_1(H\times\mathrm{id}_{\mathbb A^N})=H\pi_1$, and
\begin{eqnarray*}
&&F(H\times\mathrm{id}_{\mathbb A^N})(s_1,\ldots, s_n, x_1,\ldots, x_N)\\
&=&\sum_{j=1}^Nx_j(s_1^{c_{11}}\cdots s_n^{c_{n1}})^{w_{1j}}\cdots
(s_1^{c_{1n}}\cdots s_n^{c_{nn}})^{w_{nj}}\\
&=& \sum_{j=1}^Nx_j s_1^{\sum_{i=1}^nc_{1i}w_{ij}}\cdots
s_n^{\sum_{i=1}^nc_{ni}w_{ij}}\\
&=&s_1\Big(\sum_{j=1}^N x_j s_2^{w'_{2j}}\cdots
s_{n}^{w'_{nj}}\Big).
\end{eqnarray*}
Let $\mathbb A'^1$ be the dual affine line of $\mathbb A^1$, let
$G:\mathbb T^{n-1}\times\mathbb A^N\to\mathbb A'^1$ be the morphism
defined by
$$G(s_2,\ldots, s_{n},x_1,\ldots, x_N)=\sum_{j=1}^N x_j s_2^{w'_{2j}}\cdots
s_{n}^{w'_{nj}},$$ and let $$\langle\, ,\rangle:\mathbb
A^1\times\mathbb A'^1\to\mathbb A^1, \quad (s,\xi)\mapsto s\xi$$ be
the pairing between $\mathbb A^1$ and $\mathbb A'^1$. Denote the
restriction of $\langle\, ,\, \rangle$ to $\mathbb T^1\times\mathbb
A^{\prime 1}$ also by $\langle\, ,\, \rangle$. Then we have
$$F(H\times \mathrm{id}_{\mathbb A^N})=\langle\, ,\,\rangle\circ(\mathrm{id}_{\mathbb
T^1}\times G).$$  Moreover, we have
$$H^\ast (\mathscr K_{\chi_1}\boxtimes\cdots\boxtimes \mathscr
K_{\chi_n})\cong \mathscr K_{\chi_1'}\boxtimes\cdots\boxtimes
\mathscr K_{\chi_n'}.$$ So we have
\begin{eqnarray*}
&&\mathrm{Hyp}_\psi(\chi)\\
&=& R\pi_{2!}\Big(\pi_1^\ast(\mathscr
K_{\chi_1}\boxtimes\cdots\boxtimes \mathscr K_{\chi_n})\otimes
F^\ast\mathscr L_\psi \Big)[n+N] \\
&\cong& R(\pi_2(H\times\mathrm{id}_{\mathbb
A^N}))_!\Big((H\times\mathrm{id}_{\mathbb A^N})^\ast
\pi_1^\ast(\mathscr K_{\chi_1}\boxtimes\cdots\boxtimes \mathscr
K_{\chi_n})\otimes (H\times \mathrm{id}_{\mathbb A^N})^\ast
F^\ast\mathscr
L_\psi\Big)[n+N]\\
&\cong& R\pi_{2!}\Big(\pi_1^\ast H^\ast(\mathscr
K_{\chi_1}\boxtimes\cdots\boxtimes \mathscr
K_{\chi_n})\otimes(F(H\times\mathrm{id}_{\mathbb A^N}))^\ast\mathscr
L_\psi
\Big)[n+N]\\
&\cong& R\pi_{2!}\Big(\pi_1^\ast(\mathscr
K_{\chi_1'}\boxtimes\cdots\boxtimes \mathscr
K_{\chi_n'})\otimes(\langle\, ,\,\rangle\circ(\mathrm{id}_{\mathbb
T^1}\times G))^\ast\mathscr L_\psi \Big)[n+N].
\end{eqnarray*}
Fix notation by the following commutative diagram
$$\begin{array}{ccrcl}
\mathbb T^1\stackrel{\pi_1^{(0)}}\leftarrow &&{\mathbb
T^1}\times\mathbb T^{n-1}\times\mathbb A^N&\cong& \mathbb
T^n\times\mathbb A^N\\
&&{\scriptstyle \pi_2^{(0)}}\downarrow&&\downarrow{\scriptstyle
\pi_2}\\
&&\mathbb T^{n-1}\times\mathbb A^N&\stackrel{\pi_2^{(1)}}\to&
\mathbb
A^N\\
&&{\scriptstyle \pi_1^{(1)}}\downarrow&& \\
&&\mathbb T^{n-1}\end{array}$$ where all arrows are projections. By
the projection formula, we have
\begin{eqnarray*}
&&\mathrm{Hyp}_\psi(\chi)\\
&\cong& R\pi_{2!}\Big(\pi_1^\ast(\mathscr
K_{\chi_1'}\boxtimes\cdots\boxtimes \mathscr
K_{\chi_n'})\otimes(\langle\, ,\,\rangle\circ(\mathrm{id}_{\mathbb
T^1}\times G))^\ast\mathscr L_\psi \Big)[n+N]\\
&\cong&
R\pi^{(1)}_{2!}R\pi^{(0)}_{2!}\Big(\pi_2^{(0)\ast}\pi_1^{(1)\ast}(\mathscr
K_{\chi_2'}\boxtimes\cdots\boxtimes \mathscr K_{\chi_n'})\otimes
\pi_1^{(0)\ast}\mathscr K_{\chi'_1}\otimes(\langle\,
,\,\rangle\circ(\mathrm{id}_{\mathbb T^1}\times G))^\ast\mathscr
L_\psi \Big)[n+N]
\\
&\cong& R\pi^{(1)}_{2!}\Big(\pi^{(1)\ast}_1(\mathscr
K_{\chi_2'}\boxtimes\cdots\boxtimes \mathscr K_{\chi_{n}'})\otimes
R\pi_{2!}^{(0)}\Big(\pi_1^{(0)\ast}\mathscr
K_{\chi'_1}\otimes(\langle\, ,\,\rangle\circ(\mathrm{id}_{\mathbb
T^1}\times G))^\ast\mathscr L_\psi \Big)\Big)[n+N].
\end{eqnarray*} By Lemma 3.1 below, we have an
isomorphism $$R\pi_{2!}^{(0)}\Big(\pi_1^{(0)\ast}\mathscr
K_{\chi'_1}\otimes(\langle\, ,\,\rangle\circ(\mathrm{id}_{\mathbb
T^1}\times G))^\ast\mathscr L_\psi \Big)[1]\cong
G^\ast(FT_\psi(\kappa_!\mathscr K_{\chi'_1})),$$ where
$\kappa:\mathbb T^1\hookrightarrow \mathbb A^1$ is the canonical
open immersion. So we have
$$\mathrm{Hyp}_\psi(\chi)\cong R\pi^{(1)}_{2!}\Big(\pi^{(1)\ast}_1(\mathscr
K_{\chi_2'}\boxtimes\cdots\boxtimes \mathscr K_{\chi_{n}'})\otimes
G^\ast(FT_\psi(\kappa_!\mathscr K_{\chi'_1}))\Big)[n+N-1].$$

First consider the case where $\chi'_1$ is nontrivial. By
\cite[1.4.3]{L}, we have
$$FT_\psi(\kappa_!\mathscr K_{\chi'_1})\cong \kappa_! \mathscr
K_{\chi_1^{\prime -1}}\otimes G(\chi_1',\psi),$$ where
$G(\chi_1',\psi)$ is the rank $1$ lisse sheaf on $\mathrm {Spec}\,k$
so that the geometric Frobenius acts by multiplication by the Gauss
sum $-g(\chi'_1,\psi)=-\sum_{t\in k^\ast} \chi_1'(t)\psi(t)$, and we
denote the inverse image of $G(\chi'_1,\psi)$ on any $k$-scheme also
by $G(\chi'_1,\psi)$. So we have
$$\mathrm{Hyp}_\psi(\chi)
\cong  R\pi^{(1)}_{2!}\Big(\pi^{(1)\ast}_1(\mathscr
K_{\chi_2'}\boxtimes\cdots\boxtimes \mathscr K_{\chi_{n}'})\otimes
G^\ast \kappa_! \mathscr K_{\chi_1^{\prime -1}}\Big)\otimes
G(\chi_1',\psi)[n+N-1].$$ Let $M=q-1$. Denote by $[M]$ the
endomorphisms on $\mathbb T^{n-1}$ and on $\mathbb T^1$ defined by
$x\mapsto x^M$. Then on $\mathbb T^{n-1}$, the sheaf $[M]_\ast
\overline{\mathbb Q}_\ell$ is a direct sum of sheaves $\mathscr
K_{\chi_2'}\boxtimes\cdots\boxtimes \mathscr K_{\chi_{n}'}$, where
$\chi_i'$ $(i=2,\ldots, n)$ run over multiplicative characters of
order dividing $M$, and on $\mathbb T^1$, the sheaf
$[M]_\ast\overline{\mathbb Q}_\ell$ is a direct sum of sheaves
$\mathscr K_{\chi_1'}$, where $\chi_1'$ runs over multiplicative
characters of order dividing $M$. It follows that
$\mathrm{Hyp}_\psi(\chi)$ is a direct factor of
$R\pi^{(1)}_{2!}\Big(\pi^{(1)\ast}_1[M]_\ast \overline{\mathbb
Q}_\ell \otimes G^\ast \kappa_! [M]_\ast \overline{\mathbb
Q}_\ell\Big)\otimes G(\chi_1',\psi)[n+N-1]$. Note that the morphisms
$\pi^{(1)}_2$, $\pi^{(1)}_1$, $[M]$, $G$ and $\kappa$ are all
defined over $\mathbb Z$. So
$R\pi^{(1)}_{2!}\Big(\pi^{(1)\ast}_1[M]_\ast \overline{\mathbb
Q}_\ell \otimes G^\ast \kappa_! [M]_\ast \overline{\mathbb
Q}_\ell\Big)$ defines an object in $D_c^b(\mathbb A_{{\mathbb
Z}[1/\ell]}^N,\overline{\mathbb Q}_\ell)$. Take $\mathscr
H=R\pi^{(1)}_{2!}\Big(\pi^{(1)\ast}_1[M]_\ast \overline{\mathbb
Q}_\ell \otimes G^\ast \kappa_! [M]_\ast \overline{\mathbb
Q}_\ell\Big)[n+N-1]$. It has the required property in Theorem 0.8
(i).

Next consider the case where $\chi'_1=1$ is trivial. In this case,
by Lemma 3.2 below, we have
$$FT_\psi(\kappa_! \overline{\mathbb Q}_\ell)=
(R\kappa_{\mathbb Z[1/\ell],\ast}\overline{\mathbb
Q}_\ell)|_{\mathbb A'^1_k},$$ where $\kappa_{\mathbb
Z[1/\ell]}:\mathbb T^1_{\mathbb Z[1/\ell]}\to \mathbb A'^1_{\mathbb
Z[1/\ell]}$ is the complement of the zero section of the the (dual)
affine line $\mathbb A'^1_{\mathbb Z[1/\ell]}$ over $\mathbb
Z[1/\ell]$. So we have
$$\mathrm{Hyp}_\psi(\chi)\cong R\pi^{(1)}_{2!}\Big(\pi^{(1)\ast}_1(\mathscr
K_{\chi_2'}\boxtimes\cdots\boxtimes \mathscr K_{\chi_{n}'})\otimes
G^\ast (R\kappa_{\mathbb Z[1/\ell],\ast}\overline{\mathbb
Q}_\ell)|_{\mathbb A'^1_k}\Big)[n+N-1].$$ The same argument as above
shows that $\mathrm{Hyp}_\psi(\chi)$ is a direct factor of
$$R\pi^{(1)}_{2!}\Big(\pi^{(1)\ast}_1[M]_\ast \overline{\mathbb
Q}_\ell \otimes G^\ast(R\kappa_{\mathbb
Z[1/\ell],\ast}\overline{\mathbb Q}_\ell)|_{\mathbb
A'^1_k}\Big)[n+N-1].$$ Again the morphisms $\pi^{(1)}_2$,
$\pi^{(1)}_1$, $[M]$ and $G$ are all defined over $\mathbb Z$, and
$R\pi^{(1)}_{2!}\Big(\pi^{(1)\ast}_1[M]_\ast \overline{\mathbb
Q}_\ell \otimes G^\ast R\kappa_{\mathbb
Z[1/\ell],\ast}\overline{\mathbb Q}_\ell\Big)$ defines an object in
$D_c^b(\mathbb A_{{\mathbb Z}[1/\ell]}^N,\overline{\mathbb
Q}_\ell)$. Take $\mathscr
H=R\pi^{(1)}_{2!}\Big(\pi^{(1)\ast}_1[M]_\ast \overline{\mathbb
Q}_\ell \otimes G^\ast R\kappa_{\mathbb
Z[1/\ell],\ast}\overline{\mathbb Q}_\ell\Big)[n+N-1]$. It has the
required property in Theorem 0.8 (ii).

\begin{lemma} Notation as above. We have an
isomorphism
$$R\pi_{2!}^{(0)}\Big(\pi_1^{(0)\ast}\mathscr
K_{\chi'_1}\otimes(\langle\, ,\,\rangle\circ(\mathrm{id}_{\mathbb
T^1}\times G))^\ast\mathscr L_\psi \Big)[1]\cong
G^\ast(FT_\psi(\kappa_!\mathscr K_{\chi'_1})),$$
\end{lemma}

\begin{proof} Fix notation by the
following commutative diagram of Cartesian squares:
$$\begin{array}{crccccclc}
&&&& {\mathbb T}^1\times {\mathbb T}^{n-1}\times\mathbb A^N &&&& \\
&&&\stackrel{\mathrm{id}_{\mathbb T^1}\times G}\swarrow &&\stackrel{\kappa'}\searrow  &&& \\
&&{\mathbb T}^1\times\mathbb A'^1&&&& {\mathbb A}^1\times\mathbb
T^{n-1}\times
\mathbb A^N&& \\
&\stackrel{q_1}\swarrow &&\searrow&&\stackrel{{\mathrm{id}_{\mathbb
A^1}\times G}}
\swarrow &&\stackrel{q_2}\searrow& \\
\mathbb T^1&&&& {\mathbb A}^1\times\mathbb A'^1&&&& {\mathbb T}^{n-1}\times\mathbb A^N.\\
&\stackrel{\kappa}\searrow &&\stackrel{\scriptstyle p_1}\swarrow &&
\stackrel{p_2}\searrow && \stackrel{G}\swarrow &\\
&&{\mathbb A}^1&&&&{\mathbb A}'^1&&\\
&&&\searrow&&\swarrow&&& \\
&&&&\mathrm {Spec}\, k&&&&
\end{array}$$
By the proper base change theorem and the projection formula, we
have
\begin{eqnarray*}
&&G^\ast FT_\psi(\kappa_! \mathscr K_{\chi'_1})\\
&=& G^\ast Rp_{2!}(p_1^\ast \kappa_!\mathscr
K_{\chi'_1}\otimes\langle\, ,\,\rangle^\ast \mathscr L_\psi)[1]\\
&\cong& Rq_{2!}(\mathrm{id}_{\mathbb A^1}\times G)^\ast(p_1^\ast
\kappa_!\mathscr K_{\chi'_1}\otimes\langle\, ,\,\rangle^\ast
\mathscr
L_\psi)[1]\\
&\cong& Rq_{2!}\Big( (p_1(\mathrm{id}_{\mathbb A^1}\times G))^\ast
\kappa_!\mathscr K_{\chi'_1}\otimes (\mathrm{id}_{\mathbb A^1}\times
G)^\ast\langle\, ,\,\rangle^\ast\mathscr L_\psi\Big)[1]\\
&\cong& Rq_{2!}\Big(\kappa'_! (q_1(\mathrm{id}_{\mathbb T^1}\times
G))^\ast \mathscr K_{\chi'_1}\otimes (\mathrm{id}_{\mathbb
A^1}\times
G)^\ast\langle\, ,\,\rangle^\ast\mathscr L_\psi\Big)[1]\\
&\cong& Rq_{2!}\kappa'_!\Big((q_1(\mathrm{id}_{\mathbb T^1}\times
G))^\ast \mathscr K_{\chi'_1}\otimes \kappa'^\ast
(\mathrm{id}_{\mathbb A^1}\times
G)^\ast\langle\, ,\,\rangle^\ast\mathscr L_\psi\Big)[1]\\
&\cong&R\pi_{2!}^{(0)}\Big(\pi_1^{(0)\ast}\mathscr
K_{\chi'_1}\otimes(\langle\, ,\,\rangle\circ(\mathrm{id}_{\mathbb
T^1}\times G))^\ast\mathscr L_\psi \Big)[1].
\end{eqnarray*}
This proves our assertion.
\end{proof}

\begin{lemma} For any commutative ring $R$, let $0_R: \mathrm{Spec}\, R\to\mathbb
A'^1_R$ be the zero section of the (dual) affine line $\mathbb
A'^1_R$, and let $\kappa_R: \mathbb T^1_R\to\mathbb A'^1_R$ be the
complement of the zero section. Then we have
$$FT_\psi(\kappa_{k,!} \overline{\mathbb Q}_\ell[1])\cong
(R\kappa_{\mathbb Z[1/\ell],\ast}\overline{\mathbb
Q}_\ell[1])|_{\mathbb A'^1_k}.$$
\end{lemma}

\begin{proof} First we work over the base $k$ and omit the subscript
$k$ from our notation. We have a short exact sequence
$$0\to \kappa_! \overline{\mathbb Q}_\ell \to \overline{\mathbb Q}_\ell\to
0_\ast\overline{\mathbb Q_\ell}\to 0,$$ where
$0:\mathrm{Spec}\,k\to\mathbb A^1$ is the origin of $\mathbb A^1$.
It gives rise to a distinguished triangle
$$0_\ast \overline{\mathbb Q}_\ell\to \kappa_! \overline{\mathbb Q}_\ell[1] \to \overline{\mathbb Q}_\ell[1]\to
$$ in $D_c^b(\mathbb
A^1,\overline{\mathbb Q}_\ell)$. Taking the Deligne-Fourier
transformation, we get a distinguished triangle
$$FT_\psi(0_\ast \overline{\mathbb Q}_\ell)\to FT_\psi(\kappa_! \overline{\mathbb Q}_\ell[1])
\to FT_\psi(\overline{\mathbb Q}_\ell[1])\to
$$ in $D_c^b(\mathbb
A'^1,\overline{\mathbb Q}_\ell)$. By \cite[Proposition 1.2.3.1]{L},
we have
$$FT_\psi(0_\ast\overline{\mathbb Q}_\ell)\cong\overline{\mathbb
Q}_\ell[1],\quad FT_\psi(\overline{\mathbb Q}_\ell[1])\cong 0_\ast
\overline{\mathbb Q}_\ell(-1).$$ So we have a distinguished triangle
$$\overline{\mathbb
Q}_\ell[1]\to FT_\psi(\kappa_! \overline{\mathbb Q}_\ell[1])\to
0_\ast \overline{\mathbb Q}_\ell(-1)\to.$$ All vertices of this
distinguished triangle are perverse sheaves. So we have a short
exact sequence of perverse sheaves
$$0\to \overline{\mathbb
Q}_\ell[1]\to FT_\psi(\kappa_! \overline{\mathbb Q}_\ell[1])\to
0_\ast \overline{\mathbb Q}_\ell(-1)\to 0.$$ This shows that $
FT_\psi(\kappa_! \overline{\mathbb Q}_\ell[1])$ is an extension of
$0_\ast \overline{\mathbb Q}_\ell(-1)$ by $\overline{\mathbb
Q}_\ell[1]$ in the category of perverse sheaves. This extension is
nontrivial, that is, $FT_\psi(\kappa_! \overline{\mathbb
Q}_\ell[1])$ is not isomorphic $\overline{\mathbb
Q}_\ell[1]\bigoplus 0_\ast\overline{\mathbb Q}_\ell(-1)$. Otherwise,
by the Fourier inversion formula (\cite[Corollaire 1.2.2.3]{L}), we
would have $\kappa_! \overline{\mathbb Q}_\ell[1]\cong 0_\ast
\overline{\mathbb Q}_\ell\bigoplus \overline{\mathbb Q}_\ell[1].$
But this is not true. Thus $FT_\psi(\kappa_! \overline{\mathbb
Q}_\ell[1])$ is isomorphic to the mapping cone of a nonzero morphism
$$0_\ast \overline{\mathbb Q}_\ell(-1)[-1]\to \overline{\mathbb
Q}_\ell[1].$$ In $D_c^b({\mathbb A}^{\prime 1},\overline{\mathbb
Q}_\ell)$, we have
\begin{eqnarray*}
\mathrm{Hom}(0_\ast \overline{\mathbb Q}_\ell(-1)[-1],
\overline{\mathbb Q}_\ell[1])&\cong& \mathrm{Hom}(\overline{\mathbb
Q}_\ell(-1)[-1], R0^!\overline{\mathbb Q}_\ell[1])\\
&\cong& \mathrm{Hom}(\overline{\mathbb Q}_\ell(-1)[-1],
\overline{\mathbb Q}_\ell(-1)[-1])\\
&\cong& \overline{\mathbb Q}_\ell.
\end{eqnarray*}
It follows that any two nonzero morphisms $\phi_1,\phi_2:0_\ast
\overline{\mathbb Q}_\ell(-1)[-1]\to \overline{\mathbb Q}_\ell[1]$
differ by a nonzero scalar $\lambda\in \overline{\mathbb
Q}_\ell^\ast$. We have a commutative diagram
$$\begin{array}{ccc}
0_\ast \overline{\mathbb Q}_\ell(-1)[-1]&\stackrel{\phi_1}\to&
\overline{\mathbb Q}_\ell[1]\\
{\scriptstyle \mathrm{id}}\downarrow{\scriptstyle\cong}
&&{\scriptstyle\cong}\downarrow {\scriptstyle
\lambda}\\
0_\ast \overline{\mathbb Q}_\ell(-1)[-1]&\stackrel{\phi_2}\to&
\overline{\mathbb Q}_\ell[1].
\end{array}$$
By the axiom of triangulated categories, this diagram can be
extended to an isomorphism from the mapping cone of $\phi_1$ to the
mapping cone of $\phi_2$. Thus $FT_\psi(\kappa_! \overline{\mathbb
Q}_\ell[1])$ is isomorphic to the mapping cone of any nonzero
morphism $0_\ast \overline{\mathbb Q}_\ell(-1)[-1]\to
\overline{\mathbb Q}_\ell[1],$ or equivalently, $ FT_\psi(\kappa_!
\overline{\mathbb Q}_\ell[1])$ is isomorphic to any nontrivial
extension of $0_\ast\overline{\mathbb Q}_\ell(-1)$ by
$\overline{\mathbb Q}_\ell[1]$.

On the other hand, we have a canonical distinguished triangle
$$0_\ast R0^! \overline{\mathbb Q}_\ell\to \overline{\mathbb
Q}_\ell\to R\kappa_\ast \kappa^\ast \overline{\mathbb Q}_\ell\to.$$
It gives rise to a non-splitting short exact sequence of perverse
sheaves
$$0\to \overline{\mathbb Q}_\ell[1] \to R\kappa_\ast \overline{\mathbb Q}_\ell[1]
\to  0_\ast \overline{\mathbb Q}_\ell(-1)\to 0.$$ By the above
discussion, we must have $$FT_\psi(\kappa_! \overline{\mathbb
Q}_\ell[1])\cong  R\kappa_\ast \overline{\mathbb Q}_\ell[1].$$ One
can prove $$(R\kappa_{\mathbb Z[1/\ell],\ast}\overline{\mathbb
Q}_\ell[1])|_{\mathbb A'^1_k}\cong R\kappa_{k,\ast}\overline{\mathbb
Q}_\ell[1].$$ Our assertion follows.
\end{proof}

\end{document}